\documentclass{amsart}
\usepackage{amssymb,amsmath,amsfonts,amsthm,epsfig,amscd}
\usepackage{stmaryrd}
\usepackage[all,cmtip,poly]{xy}
\usepackage{color}
\usepackage[bookmarksopen,bookmarksdepth=2]{hyperref}

\newcommand{\barPhi    }{\overline{\Phi}}
\newcommand{\bare}{\overline{{e}}}
\newcommand{\barM}{\overline{{M}}}
\newcommand{\gr}{{\operatorname{gr}\,}}
\DeclareMathOperator{\FL}{FL}
\newcommand{\Rep}{{\operatorname{Rep}}}
\newcommand{\chibar}{\overline{\chi}}
\newcommand{\Vbar}{\overline{V}}
\newcommand{\Ubar}{\overline{U}}
\newcommand{\Xbar}{\overline{X}}

 \newcommand{\sigmabar   }{\overline{\sigma}}
\newcommand{\tv}{\tilde{v}}

\newcommand{\Zpbartimes}{\overline{\mathbb{Z}}_p^\times}

\newcommand{\kappabar}{\overline{\kappa}}

\newcommand{\mc}{\mathcal}

\newcommand{\psibar}{\overline{\psi}}

\newtheorem{ithm}{Theorem}

\newtheorem{thm}[subsubsection]{Theorem}

\newtheorem{lem}[subsubsection]{Lemma}

\newtheorem{cor}[subsubsection]{Corollary}

\newtheorem{prop}[subsubsection]{Proposition}

\theoremstyle{definition}

\newtheorem{defn}[subsubsection]{Definition}

\theoremstyle{remark}
\newtheorem{remark}[subsubsection]{Remark}
\newtheorem{rem}[subsubsection]{Remark}

\newtheorem{example}[subsubsection]{Example}
\newtheorem{examples}[subsubsection]{Examples}

\def\numequation{\addtocounter{subsubsection}{1}\begin{equation}}
\def\nummultline{\addtocounter{subsubsubsection}{1}\begin{multline}}
\def\anumequation{\addtocounter{subsection}{1}\begin{equation}}

\newif\iffinalrun
\finalrunfalse
\iffinalrun
\else
\fi

\iffinalrun
  \newcommand{\need}[1]{}
  \newcommand{\mar}[1]{}
\else
  \newcommand{\need}[1]{{\tiny *** #1}}
  \newcommand{\mar}[1]{\marginpar{\raggedright\tiny #1}}
\fi

\newcommand{\A}{\AA}

\newcommand{\F}{\FF}

\newcommand{\Q}{\QQ}

\newcommand{\Z}{\ZZ}

\newcommand{\m}{\frakm}

\renewcommand{\AA}{{\mathbb A}}

\newcommand{\FF}{{\mathbb F}}

\newcommand{\QQ}{{\mathbb Q}}

\newcommand{\ZZ}{{\mathbb Z}}

\renewcommand{\bf}{\ensuremath{\mathbf{f}}}

\newcommand{\cG}{{\mathcal G}}

\newcommand{\cO}{{\mathcal O}}

\newcommand{\frakm}{\mathfrak{m}}

\newcommand{\Fbar}{\overline{\F}}
\newcommand{\Qbar}{\overline{\Q}}
\newcommand{\Zbar}{\overline{\Z}}
\renewcommand{\o}{\overline}

\newcommand{\Fp}{\F_p}

\newcommand{\Fpbar}{\Fbar_p}

\newcommand{\Zp}{\Z_p}

\newcommand{\Zpbar}{\Zbar_p}

\newcommand{\Qp}{{\Q_p}}

\newcommand{\Qpbar}{\Qbar_p}

\DeclareMathOperator{\coker}{coker}

\DeclareMathOperator{\Ext}{Ext}
\DeclareMathOperator{\Fil}{Fil}

\DeclareMathOperator{\GL}{GL}

\DeclareMathOperator{\Hom}{Hom}
\DeclareMathOperator{\im}{im}
\DeclareMathOperator{\Ind}{Ind}

\DeclareMathOperator{\rank}{rank}

\DeclareMathOperator{\WD}{WD}

\newcommand{\Frob}{\mathrm{Frob}}

\newcommand{\HT}{\mathrm{HT}}

\newcommand{\nr}{\mathrm{nr}}

\newcommand{\ur}{\mathrm{ur}}

\newcommand{\rhobar}{\overline{\rho}}

\newcommand{\into}{\hookrightarrow}
\newcommand{\onto}{\twoheadrightarrow}

\newcommand{\rbar}{\overline{r}}

\newcommand{\cyclo}{\varepsilon}

\newcommand{\Mbar}{\overline{M}}
\newcommand{\fbar}{\overline{f}}

\begin{document}
\title{Potentially crystalline lifts of certain prescribed types}

\author[T.\ Gee]{Toby Gee} \email{toby.gee@imperial.ac.uk} \address{Department of
  Mathematics, Imperial College London,
  London SW7 2AZ, UK}

\author[F.\ Herzig]{Florian Herzig}
\email{herzig@math.toronto.edu}
\address{Department of Mathematics, University of Toronto}

\author[T.\ Liu]{Tong
  Liu}\email{tongliu@math.purdue.edu}\address{Department of
  Mathematics, Purdue University}

\author[D.\ Savitt]{David Savitt} \email{savitt@math.jhu.edu}
\address{Department of Mathematics, Johns Hopkins University}
\thanks{The first author was
  partially supported by a Leverhulme Prize, EPSRC grant EP/L025485/1, Marie Curie Career
  Integration Grant 303605, and by
  ERC Starting Grant 306326.}

\thanks{The second author was partially supported by
  a Sloan Fellowship and an NSERC grant}

\thanks{The third author was partially supported by NSF grants DMS-0901360 and DMS-1406926}

\thanks{The fourth author was partially supported by NSF grant
  DMS-0901049 and  NSF CAREER grant DMS-1054032}

\maketitle

\begin{abstract}
  We prove several results concerning the existence of potentially crystalline lifts
  of prescribed Hodge--Tate weights and inertial types of a given
  representation $\rbar:G_K\to\GL_n(\Fpbar)$, where $K/\Qp$ is a
  finite extension. Some of these results are proved by purely local
  methods, and are expected to be useful in the application of
  automorphy lifting theorems. The proofs of the other results are
  global, making use of automorphy
  lifting theorems.
\end{abstract}

\section{Introduction}\label{sec:intro} Let $p$ be a prime, let $K/\Qp$ be a finite extension, and let
$\rbar:G_K\to\GL_n(\Fpbar)$ 
  be a continuous representation. For many reasons,
it is a natural and important question to study the lifts of $\rbar$ to de
Rham representations $r:G_K\to\GL_n(\Zpbar)$; for example, the de Rham lifts
of fixed Hodge and inertial types are parameterised by a universal (framed)
deformation ring thanks to~\cite{kisindefrings}, and the study of these
deformation rings is an important step in proving automorphy lifting theorems,
going back to Wiles' proof of Fermat's Last Theorem, which made use of
Ramakrishna's work on flat deformations~\cite{MR1935843}. 

It is therefore slightly vexing that (as far as we are aware) it
is currently an open problem to prove that for a general choice of
$\rbar$, a single such lift $r$ exists (equivalently, to show for each
$\rbar$ that at least one of Kisin's deformation rings is
nonzero). Some results in this direction can be found
in the Ph.D.\ thesis of Alain Muller \cite{MullerThesis}.
This note sheds little further light on this question, but
rather investigates the question of congruences between de Rham
representations of different Hodge and inertial types; that is, in
many of our results we
suppose the existence of a single lift, and see what other lifts (of
differing Hodge and inertial types) we can produce from this.
The existence of congruences between
representations of differing such types is conjecturally governed by the
(generalised) Breuil--M\'ezard conjecture (at least for regular Hodge
types; see~\cite{emertongeerefinedBM}). This
conjecture is almost completely open beyond the case of $\GL_2/\Qp$, so it is of
interest to prove unconditional results.

We prove several such results in this paper, by a variety of different
methods. Some of our
results make use of the notion of a \emph{potentially diagonalisable} Galois
representation, which was introduced in~\cite{BLGGT}, and is very important in
automorphy lifting theorems. It is expected (\cite[Conj.~A.3]{emertongeerefinedBM}) that every $\rbar$ admits a
potentially diagonalisable lift of regular weight\footnote{Recall that a de Rham representation of $G_K$ is said to have
\emph{regular} weight if for any continuous embedding $K\into\Qpbar$, the corresponding Hodge--Tate
weights are all distinct.}, but this is at present known 
only if $n\le 3$ or $\rbar$ is semisimple; 
see for
example~\cite[Lem.~2.2]{CEGGPSBreuilSchneider}, and the proof
of~\cite[Prop.\ 2.5.7]{MullerThesis} for the case $n=3$.  It seems
plausible that these arguments could be extended to cover other small dimensions, but the case of general~$n$
seems to be surprisingly difficult.

We recall that an $n$-dimensional de Rham representation of $G_K$ is said to have Hodge
type $0$ if  for any continuous embedding $K\into\Qpbar$ the corresponding Hodge--Tate
weights are $0,1,\ldots,n-1$; while if $K/\Qp$ is unramified, a
crystalline representation of $G_K$ is said to be
Fontaine--Laffaille if for each continuous embedding $K\into\Qpbar$ the
corresponding Hodge--Tate weights are all contained in an
interval of the form $[i,i+p-2]$. 
We remark that we will normalise Hodge--Tate
weights so that the cyclotomic character $\cyclo$ has Hodge--Tate
weight $-1$.

 Our first result is the following theorem, 
which will be used
in forthcoming work of Arias de Reyna and Dieulefait. 

\begin{ithm}\label{thmA}
\emph{(Cor.~\ref{cor:ld-sadr})} Suppose that
$K/\Qp$ is unramified, and fix an integer $n\ge 1$. Then there is a
finite extension $K'/K$, depending only on $n$ and $K$, with the following
property: if $\rbar:G_K\to\GL_n(\Fpbar)$ has a
Fontaine--Laffaille lift, then it also has a potentially
diagonalisable lift $r:G_K\to\GL_n(\Zpbar)$ of Hodge type $0$ with the property that
$r|_{G_{K'}}$ is crystalline.
\end{ithm}

In fact this is a special case of a result
(Cor.~\ref{cor: peu ramimplies HT type whatever and bounded tame
  type}) that holds for a
more general class of representations $\rbar$ that we call \emph{peu
  ramifi\'ee}, and with no assumption that the finite extension
$K/\Qp$ is unramified.  We expect that this result should even be true
without the assumption that $\rbar$ is peu ramifi\'ee,
but we do not
know how to prove this; indeed, as mentioned above, we do not know how
to produce a single de Rham lift in general!

To explain why this result is reasonable, and to give some indication
of the proof, we focus on the case that~$K=\Qp$ and~$n=2$. Assume for
simplicity in the following discussion that $p>2$.  One way to
see that we should expect the result to be true (at least if we remove
``potentially diagonalisable'' from the statement) is that it is then the local Galois
analogue of the well-known statement that every modular eigenform of
level prime to~$p$ is congruent to one of weight~$2$ and bounded level
at~$p$. Indeed, via the mechanism of modularity lifting theorems and
potential modularity, it is possible to turn this analogy into a proof.
(See Theorem~\ref{thmC} below. Since all potentially Barsotti--Tate
representations are known to be potentially diagonalisable, this
literally proves Theorem~\ref{thmA} in this case, but this deduction
cannot be made if~$n>2$.)

Since these global methods are (at least at present) unable to handle
the case $n>2$, a local approach is needed, which we again motivate via
the case $K=\Qp$ and $n=2$. 
The
possible~$\rbar:G_{\Qp}\to\GL_2(\Fpbar)$ are well-understood; they are
either irreducible representations, in which case they are induced
from characters of the unramified quadratic extension~$\Q_{p^2}$
of~$\Qp$, or they are reducible, and are extensions of unramified
twists of powers of the mod~$p$ cyclotomic character~$\omega$.

In the first case, the representations are induced from characters
of~$G_{\Q_{p^2}}$ which become unramified after restriction to any
totally ramified extension of degree $p^2-1$,
and it is
straightforward to produce the required lifts by considering
inductions of potentially crystalline characters of~$G_{\Q_{p^2}}$
which become crystalline over such an extension; see Lemma~\ref{lem:existence of a bounded lift for an irreducible rep over
    an unramified field}. Such representations are automatically
  potentially diagonalisable, as after restriction to some finite
  extension they are even a direct sum of crystalline characters.

This leaves the case that~$\rbar$ is reducible. After twisting, we may
assume 
that~$\rbar$ is an extension of an
unramified twist of~$\omega^{-i}$ by the trivial character, for some
$0\le i\le p-2$. 
Then the
natural way to lift to characteristic zero and Hodge type $0$ is to
try to lift to an extension of an unramified
twist of~$\cyclo^{-1} \widetilde{\omega}^{1-i}$ by the trivial character, where~$\widetilde{\omega}$ is the
Teichm\"uller lift of~$\omega$; this is promising because
any such extension is at least potentially semistable, and becomes
semistable over~$\Qp(\zeta_p)$ (which is in particular independent of
the specific reducible~$\rbar$ under consideration), and if it is
potentially crystalline, then it is also potentially diagonalisable
(as it is known that any successive extension of characters which is
potentially crystalline is also potentially diagonalisable).

The problem of producing such lifts is one of Galois cohomology, and
Tate's duality theorems show that when~$i\ne 1$ there is no
obstruction to lifting.\footnote{%
  This is true
  even for the ramified
  self-extensions of the trivial character in the case $i=0$, which are not
  Fontaine--Laffaille, although they are peu ramifi\'ee in the sense
  of this paper (Definition~\ref{def:peu-ramif}).}
 It is also easy to check that in this case  the
lifts are automatically potentially crystalline. However, when~$i=1$ the
situation is more complicated.
 Then one can check that
tr\`es ramifi\'ee extensions of $\omega^{-1}$ by the trivial character do
not lift to extensions of a non-trivial
unramified twist of~$\cyclo^{-1}$ by the trivial character, but only lift to
semistable non-crystalline extensions of $\cyclo^{-1}$ by  the trivial character. However, this is the only obstruction to carrying out
the strategy in this case; and in fact, since tr\`es ramifi\'ee representations
do not have Fontaine--Laffaille lifts, the result also follows in the
case~$i=1$.

We prove Theorem~\ref{thmA} by a generalisation of this strategy:\ we
write~$\rbar$ as an extension of irreducible representations, lift the
irreducible representations as inductions of crystalline characters,
and then lift the extension classes. However, the issues that arose in
the previous paragraph in the case~$i=1$ are more complicated
in general. 
To
address this, we make use of the following observation:\ in the case
considered in the previous paragraphs (that is, $K=\Qp$, $n=2$,
and~$\rbar$ has a trivial subrepresentation),
if $\rbar$ is not tr\`es ramifi\'ee then it admits ``many'' reducible
crystalline lifts; indeed,
it can be lifted as an extension by the trivial character of
\emph{any} unramified twist of~$\cyclo^{-i}$ that lifts the
corresponding character mod~$p$.

This freedom to twist by unramified characters is in marked contrast
to the behaviour in the tr\`es ramifi\'ee case, and can be exploited in the
Galois cohomology calculations used to produce the potentially
crystalline lifts of Hodge type~$0$. Motivated by these observations, we
introduce a generalisation (Definition~\ref{def:peu-ramif}) of the
classical notion of peu ramifi\'ee representations, and we prove by
direct Galois cohomology arguments that the peu ramifi\'ee condition
allows great flexibility in the production of lifts to
varying reducible representations (see Theorem~\ref{thm: existence of lifts, prescribed on inertia, for
    peu ramifiee reps} and Corollary~\ref{cor: peu ramimplies HT type whatever and bounded tame type}).

Conversely, every representation that
  admits enough lifts of the sort promised by Theorem~\ref{thm: existence of lifts, prescribed on inertia, for
    peu ramifiee reps} must in fact be peu ramifi\'ee (see
  Proposition~\ref{prop:pr-htl} for a precise statement); such a
  representation is said to admit ``highly twisted lifts.'' 
  We show
that representations that admit 
Fontaine--Laffaille lifts also
admit highly twisted lifts
(Proposition~\ref{prop:fllift}), and so deduce that Corollary~\ref{cor: peu ramimplies HT type whatever and bounded tame type} applies whenever the residual representation is
  Fontaine--Laffaille.  Theorem~\ref{thmA} follows.

Using roughly the same purely local methods, we additionally prove the
following.

\begin{ithm}\label{ithm:crys-lift}\emph{(Cor.\ \ref{cor:crys-lift-of-serre-weight})}
    Suppose that $\rbar: G_K
    \to \GL_n(\Fpbar)$ is peu ramifi\'ee. Then~$\rbar$ has a
    crystalline lift of some Serre weight {\upshape(}in the sense of Section~\ref{sec:cryst-lifts-font}{\upshape)}.
\end{ithm}

In contrast to these relatively concrete local arguments, in Section~\ref{sec:global lifts} we use global methods,
and in particular the
potential automorphy machinery of~\cite{BLGGT}. Our first result is
the following, which takes as input a potentially crystalline lift that
could have highly ramified inertial type, or highly spread out
Hodge--Tate weights, and produces a crystalline lift of small
Hodge--Tate weights.

\begin{ithm}\label{thmC} \emph{(Thm.\ \ref{thm:pot diag implies serre weight})}
  Suppose that $p\nmid 2n$, and that $\rbar : G_K \to \GL_n(\Fpbar)$ has a potentially diagonalisable
  lift of some regular weight. Then the following hold.
  \begin{enumerate}

  \item There exists a finite extension $K'/K$ {\upshape(}depending only on $n$
    and $K$, and not on $\rbar${\upshape)} such that $\rbar$ has a lift $r : G_K
    \to \GL_n(\Zpbar)$ of Hodge type $0$ that becomes crystalline over $K'$.

\item  $\rbar$ has a crystalline lift of some
  Serre weight.

  \end{enumerate}
\end{ithm}

The first part of this result should be contrasted with
Theorem~\ref{thmA} above, while the second part should be contrasted with
Theorem~\ref{ithm:crys-lift}. For instance, we remark that it follows from
Theorem~\ref{thmA} (or more precisely, from its more general statement
for peu ramifi\'ee representations) that every peu ramifi\'ee
representation~$\rbar$ admits a potentially diagonalisable lift of some
regular weight, whereas this latter condition on $\rbar$  is an input to Theorem~\ref{thmC}.

If $K/\Qp$ is unramified and $\rbar$ admits a lift of extended
FL weight (see Section~\ref{sec:cryst-lifts-font} for
this terminology), we also show the following ``weak Breuil--M\'ezard
result''.

\begin{ithm}\label{thmD}
  \emph{(Thm.\ \ref{thm:FL lift implies lots more lifts})} Suppose
  that $p\ne n$, that $K/\Qp$ is
  unramified, and that $\rbar:  G_K \to \GL_n(\Fpbar)$ has a crystalline lift
  of some extended FL weight~$F$. 
  If $F$ is a
  Jordan--H\"older factor of~$\sigmabar(\lambda,\tau)$ for some $\lambda$,
  $\tau$, then $\rbar$ has a potentially crystalline lift of type~$(\lambda,\tau)$.
\end{ithm}
 Since there is no restriction on~$\lambda$ or~$\tau$, this result seems
to be well beyond anything that can currently be proved directly using integral $p$-adic
Hodge theory.

If we knew that all potentially crystalline lifts were potentially
diagonalisable, then the special case of Theorem~\ref{thmA} in which the given  Fontaine--Laffaille lift is
regular would be an easy
consequence of part (1) of Theorem~\ref{thmC} (note that the existence of a
regular Fontaine--Laffaille lift implies that  $p > n$). However, we do
not know how to prove that general potentially crystalline representations are
potentially diagonalisable (and we do not have any strong
evidence that it should be true). 

\subsection{Acknowledgements}We would like to thank Luis Dieulefait for asking a
question which led to us writing this paper, as well as Alain Muller
for valuable discussions.

\subsection{Notation and conventions}\label{sec:notation} 
Fix a prime~$p$, and let~$K/\Qp$ be a finite extension with ring of
integers $\cO_K$. Write~$G_K$ for the
absolute Galois group of~$K$, $I_K$ for the inertia subgroup
of~$G_K$, and~$\Frob_K \in G_K$ for a choice of geometric Frobenius.  All representations of $G_K$ are assumed without further
comment to be continuous.  Write $v_K$ for the $p$-adic valuation on
$K$ taking the value $1$ on a uniformiser of $K$, as well as
for the unique extension of this valuation to any algebraic extension
of~$K$.

\subsubsection{Inertial types}
\label{sec:inertial-types}
An \emph{inertial type} is   a representation
$\tau:I_K\to\GL_n(\Qpbar)$ with open kernel  which extends
to the Weil group $W_K$. We say that a de Rham representation
$r:G_K\to\GL_n(\Qpbar)$ has inertial type $\tau$ if the
restriction to $I_K$ of the Weil--Deligne representation $\WD(r)$ associated to
$r$ is equivalent to $\tau$. Given an inertial type $\tau$, there
is a (not necessarily unique) finite-dimensional smooth irreducible $\Qpbar$-representation
$\sigma(\tau)$ of $\GL_n(\cO_K)$ associated to $\tau$ by the
``inertial local Langlands correspondence'', which we normalise as
in~\cite[Conj.\ 4.1.3]{emertongeerefinedBM}. (Note that there is an unfortunate
difference in conventions between this and that of~\cite[Thm.\
4.1.5]{emertongeerefinedBM}, but it is this normalisation that is used
in the remainder of~\cite{emertongeerefinedBM}.) 
We can and do suppose that
$\sigma(\tau)$ is defined over~$\Zpbar$.

\subsubsection{Hodge--Tate weights and Hodge types}
\label{sec:hodge-tate}
If $W$ is a de Rham
representation of $G_K$ over $\Qpbar$, and  $\kappa:K \into \Qpbar$,
then we will write $\HT_\kappa(W)$ for the multiset of Hodge--Tate
weights of $W$ with respect to $\kappa$.  By definition, 
the multiset $\HT_\kappa(W)$ contains $i$ with multiplicity
$\dim_{\Qpbar} (W \otimes_{\kappa,K} \widehat{\overline{K}}(i))^{G_K}
$. Thus for example if $\cyclo$ denotes the $p$-adic cyclotomic character
of~$G_K$, then $\HT_\kappa(\cyclo)=\{ -1\}$ for all $\kappa$. 

We say that $W$ has \emph{regular} Hodge--Tate weights if for each
$\kappa$, the elements of $\HT_\kappa(W)$ are pairwise distinct. Let $\Z^n_+$
denote the set of tuples $(\lambda_1,\dots,\lambda_n)$ of integers
with $\lambda_1\ge \lambda_2\ge\dots\ge \lambda_n$. Then if $W$ has
regular Hodge--Tate weights, there is a unique
$\lambda=(\lambda_{\kappa,i})\in(\Z^n_+)^{\Hom_\Qp(K,\Qpbar)}$ such that for each $\kappa:K\into
\Qpbar$, \[\HT_{\kappa}(W)=\{\lambda_{\kappa,1}+n-1,\lambda_{\kappa,2}+n-2,\dots,\lambda_{\kappa,n}\},\]
and we say that $W$ is regular of \emph{Hodge type} $\lambda$.

\subsubsection{Representations of $\GL_n$ and Serre weights}
\label{sec:repr-gl_n-serre}
For any
$\lambda\in\Z^n_+$, view $\lambda$ as a dominant weight (with respect to the
upper triangular Borel subgroup) of the
algebraic group $\GL_{n}$ in the usual way, and let $M'_\lambda$
be the algebraic $\cO_K$-representation of $\GL_n$ given
by \[M'_\lambda:=\Ind_{B_n}^{\GL_n}(w_0\lambda)_{/\cO_K}\] where $B_n$
is the Borel subgroup of upper-triangular matrices of $\GL_n$, and
$w_0$ is the longest element of the Weyl group (see \cite{MR2015057}
for more details of these notions, and note that $M'_\lambda$ has
highest weight $\lambda$). Write $M_\lambda$ for the
$\cO_K$-representation of $\GL_n(\cO_K)$ obtained by evaluating
$M'_\lambda$ on $\cO_K$. For any $\lambda\in(\Z^n_+)^{\Hom_\Qp(K,\Qpbar)}$
we write $L_\lambda$ for the $\Zpbar$-representation of $\GL_n(\cO_K)$
defined by \[L_{\lambda}:=\otimes_{\kappa:K\into
  \Qpbar}M_{\lambda_\kappa}\otimes_{\cO_K,\kappa}\Zpbar.\]

Let $k$ be the residue field of~$K$. We call isomorphism classes of irreducible $\Fpbar$-representations of~$\GL_n(k)$ \emph{Serre weights};
they can be parameterised as follows. We say that an element $(a_i)$ of $\Z^n_+$
is \emph{$p$-restricted} if $p-1 \ge a_i - a_{i+1}$ for all $1\le i \le n-1$,
and we write $X_1^{(n)}$ for the set of $p$-restricted elements. Given any
$a\in X_1^{(n)}$, we define the $k$-representation $P_a$ of $\GL_n(k)$ to be
the representation obtained by evaluating $\Ind_{B_n}^{\GL_n}(w_0 a)_{/k}$ on
$k$, and let $N_a$ be the irreducible sub-$k$-representation of $P_a$ generated
by the highest weight vector (that this is indeed irreducible follows from the
analogous result for the algebraic group $\GL_n$, cf.\ II.2.2--II.2.6 in \cite{MR2015057},
and the appendix to \cite{herzigthesis}).

If $a = (a_{\kappabar,i}) \in (X_1^{(n)})^{\Hom(k,\Fpbar)}$, write $a_{\kappabar}$ for the
component of $a$ indexed by $\kappabar \in \Hom(k,\Fpbar)$. If $a\in (X_1^{(n)})^{\Hom(k,\Fpbar)}$ then we define an
irreducible $\Fpbar$-representation $F_a$ of $\GL_n(k)$
by \[F_a:=\otimes_{\kappabar\in\Hom(k,\Fpbar)}N_{a_{\kappabar}}\otimes_{k,\kappabar}\Fpbar.\]
The representations $F_a$ are irreducible,  and every Serre weight is
(isomorphic to one) of the
form $F_a$ for some $a$. The choice of~$a$ is not unique:\ one has
$F_a \cong F_{a'}$ if and only if there
exist integers $x_{\kappabar}$ such that $a_{\kappabar,i} - a'_{\kappabar,i} =
x_{\kappabar}$ for all $\kappabar,i$ and, for any
labeling $\kappabar_j$ of the elements of $\Hom(k,\Fpbar)$ such that $\kappabar_j^p =
\kappabar_{j+1}$ we have $\sum_{j=0}^{f-1} p^j x_{\kappabar_j} \equiv 0
\pmod{p^f-1}$, where $f = [k:\Fp]$.  In this case we write $a \sim a'$.

We remark that if $K/\Qp$ is unramified and $a\in (X_1^{(n)})^{\Hom(k,\Fpbar)}$ 
satisfies $a_{\kappabar,1}-a_{\kappabar,n}\le p-(n-1)$ for each~$\kappabar$, then
$L_a\otimes_{\Zpbar}\Fpbar \cong F_a$ as representations of $\GL_n(\cO_K)$. The reason
is that $P_b = N_b$ whenever $b \in \Z^n_+$ satisfies $b_1-b_n \le p-(n-1)$ (cf.\ \cite[II.5.6]{MR2015057}).

\subsubsection{Potentially crystalline representations 
}
\label{sec:cryst-lifts-font}
An element
  $\lambda\in(\Z^n_+)^{\Hom_{\Qp}(K,\Qpbar)}$ is said to be a \emph{lift} of an element
  $a\in(X_1^{(n)})^{\Hom(k,\Fpbar)}$ if for each $\kappabar \in \Hom(k,\Fpbar)$ there exists
  $\kappa_{\kappabar} \in \Hom_{\Qp}(K,\Qpbar)$ lifting $\kappabar$
  such that $\lambda_{\kappa_{\kappabar}}=a_{\kappabar}$, and $\lambda_{\kappa'} =
  0$ for all other $\kappa'\neq \kappa_{\kappabar}$
  in $\Hom_{\Qp}(K,\Qpbar)$ lifting $\kappabar$.  If $\lambda$ is a lift of $a$, then $F_a$ is a Jordan--H\"older factor of
  $L_\lambda\otimes\Fpbar$.

Given a pair $(\lambda,\tau)$, we say that a potentially crystalline representation~$W$ of~$G_K$
over $\Qpbar$ has type $(\lambda,\tau)$ if it is regular of Hodge type~$\lambda$, and has inertial
type~$\tau$. Write $\sigma(\lambda,\tau)$ for $L_\lambda\otimes_{\Zpbar}\sigma(\tau)$, a
$\Zpbar$-representation of~$\GL_n(\cO_K)$, and write $\sigmabar(\lambda,\tau)$ for
the semisimplification of $\sigma(\lambda,\tau)\otimes_{\Zpbar}\Fpbar$. Then the action of~$\GL_n(\cO_K)$
on~$\sigmabar(\lambda,\tau)$ factors through~$\GL_n(k)$, so that the
Jordan--H\"older factors of~$\sigmabar(\lambda,\tau)$ are Serre weights.

If $\rbar : G_K \to \GL_n(\Fpbar)$ has a crystalline lift $W$ of type $(\lambda,1)$ (that is, $W$ is crystalline of Hodge type~$\lambda$),
and $\lambda$ is a lift of some~$a \in(X_1^{(n)})^{\Hom(k,\Fpbar)}$, then we say that~$\rbar$ has a
\emph{crystalline lift
  of Serre weight~$F_a$}. This terminology is sensible because
the existence of a crystalline lift of Hodge type $\lambda$ for some
lift $\lambda$ of $a$ does not depend on the choice of the element $a$ in its equivalence
class under the equivalence relation $\sim$ (cf.~\cite[Lem.~7.1.1]{EGHS}).

If
furthermore $K/\Qp$ is unramified, and $a_{\kappabar,1}-a_{\kappabar,n}\le p-1-n$ for
all~$\kappabar$, then we say that~$a$ (or $F_a$) is an \emph{FL
  weight}, and that~$\rbar$  \emph{has a crystalline lift of
  FL weight~$F_a$}. If instead $a_{\kappabar,1}-a_{\kappabar,n}\le p-n$ for
all~$\kappabar$, then we say that~$a$ (or $F_a$) is an \emph{extended
  FL weight}, and that~$\rbar$ has a \emph{crystalline lift of extended
  FL weight~$F_a$}. 

\subsubsection{Potential diagonalisability}
\label{sec:potent-diag}
Following~\cite{BLGGT}, we say that a potentially crystalline
representation~$r:G_K\to\GL_n(\Zpbar)$ with distinct Hodge--Tate weights is \emph{potentially
  diagonalisable} if for some finite extension $K'/K$,
$r|_{G_{K'}}$ is crystalline, and the corresponding~$\Qpbar$ point
of the corresponding crystalline deformation ring lies on the same
irreducible component as some direct sum of crystalline
characters. (For example, it follows from the main theorem
of~\cite{GaoLiu12}
that any crystalline representation of extended FL
weight is potentially diagonalisable.)

\section{Local existence of lifts in the residually
  Fontaine--Laffaille case}\label{sec:FL lifts to pot
  diag}

\subsection{Peu ramifi\'ee representations}
\label{sec:peu-ramif-repr}

Recall that for any discrete $G_K$-module $X$, the space
$H^1_{\ur}(G_K,X)$ of unramified classes in $H^1(G_K,X)$ is the kernel of
the restriction map $H^1(G_K,X) \to H^1(I_K,X)$; by the
inflation-restriction sequence, this is the same as the image of
the inflation map $H^1(G_K/I_K,X^{I_K}) \into H^1(G_K,X)$. We will
make use of the
following well-known fact.

\begin{lem}
  \label{lem:dim-of-ur}
Suppose that $X$ is a discrete $G_K$-module that is moreover a
finite-dimensional vector space over a field $\F$.  Then
 $$\dim_{\F} H^1_{\ur}(G_K,X) = \dim_{\F} H^0(G_K,X).$$
\end{lem}

\begin{proof}
  We have $$\dim_{\F} H^1(G_K/I_K,X^{I_K}) = \dim_{\F}
  H^0(G_K/I_K,X^{I_K}) = \dim_{\F} H^0(G_K,X),$$
the first equality coming from the fact that $H^i(G_K/I_K,X^{I_K})$
for $i=  0,1$ are, respectively, the invariants and co-invariants of
$X^{I_K}$ under $\Frob_K-1$.
\end{proof}

\begin{defn}
 \label{def:peu-ramif-fil}
Suppose that $K/\Qp$ is a
finite extension and $\F$ is a field of characteristic $p$.  Consider a representation
  $\rbar:G_K\to\GL_n(\F)$, let $\Vbar$ be the underlying
  $\F[G_K]$-module of $\rbar$, and suppose that $0 = \Ubar_0
  \subset \Ubar_1 \subset \dots \subset \Ubar_{\ell} = \Vbar$ is an
  increasing filtration on $\Vbar$ by $\F[G_K]$-submodules. Write
  $\Vbar_i := \Ubar_i/\Ubar_{i-1}$. We say
  that $\rbar$ is \emph{peu ramifi\'ee with respect to the
  filtration $\{\Ubar_i\}$} if for all $1 \le i \le \ell$ the class in
  $H^1(G_K,\Hom_{\F}(\Vbar_i,\Ubar_{i-1}))$
   defined by $\Ubar_i$
  (regarded as an extension of $\Vbar_i$ by $\Ubar_{i-1}$) is annihilated under Tate local
  duality by
  $H^1_{\ur}(G_K,\Hom_{\F}(\Ubar_{i-1},\Vbar_i(1)))$.
\end{defn}

Since  group cohomology is compatible with base change for field
extensions, so is Definition~\ref{def:peu-ramif-fil}:\ that is, if
$\F'/\F$ is any field extension, then
$\rbar$ is peu ramifi\'ee with respect to some filtration
$\{\Ubar_i\}$ if and only if $\rbar
\otimes_{\F} \F'$ is peu ramifi\'ee with respect to the filtration
$\{\Ubar_i \otimes_{\F} \F'\}$.

Definition~\ref{def:peu-ramif-fil} is
most interesting in the case where the filtration
$\{\Ubar_i\}$ is \emph{saturated}, i.e., where the graded pieces $\Vbar_i$ are
irreducible.   (For instance, any $\rbar$ will trivially be peu ramifi\'ee with respect to the one-step filtration $0 = \Ubar_0 \subset
\Ubar_1 = \Vbar$.) This motivates the following further definition.

\begin{defn}
  \label{def:peu-ramif} We say that $\rbar$ is \emph{peu ramifi\'ee}
  if there exists a saturated filtration $\{\Ubar_i\}$ with respect to
  which $\rbar$ is peu ramifi\'ee as in
  Definition~\ref{def:peu-ramif-fil}. 
\end{defn}

\begin{examples} \hfill \label{examples-pr}
  \begin{enumerate}
  \item
If $n = 2$  and $\rbar \cong
\begin{pmatrix}
  \chi \omega & * \\ 0 & \chi
\end{pmatrix}$ for some character $\chi$, then
Definition~\ref{def:peu-ramif} coincides with the
usual definition of peu ramifi\'ee. (Recall that $\omega$ denotes the mod~$p$ cyclotomic character.) Indeed, the duality pairing
$H^1(G_K,\Fp(1)) \times H^1(G_K,\Fp) \to \Qp/\Zp$ can be identified
(via the Kummer and Artin maps) with the evaluation map
$$K^{\times}/(K^{\times})^p \times  \Hom(K^{\times},\Fp) \to \Fp \into
\Qp/\Zp,$$
from which it is immediate that the classes in $H^1(G_K,\Fp(1))$
that are annihilated by $H^1_{\ur}(G_K,\Fp)$ are precisely those which
are identified with $\cO_K^{\times}/(\cO_K^{\times})^p$ by the Kummer map.

\item If $\rbar$ is semisimple then trivially $\rbar$ is peu ramifi\'ee.

\item  If there are no nontrivial $G_K$-maps $\Ubar_{i-1}
  \to \Vbar_i(1)$ for any $i$ (e.g.\ if one has $\Vbar_j \not\cong \Vbar_i(1)$ for all $j < i$) then
  $\rbar$ is necessarily peu ramifi\'ee because by
  Lemma~\ref{lem:dim-of-ur}  we have
  $H^1_{\ur}(G_K,\Hom_{\F}(\Ubar_{i-1},\Vbar_i(1)))=0$.

\item Suppose $K/\Qp$ is unramified. We will prove in Section~\ref{subsec: FL
  representations are universally twisted} that Fontaine--Laffaille
representations are peu ramifi\'ee, so that all of the main results in this
section will apply to Fontaine--Laffaille representations.
  \end{enumerate}

\end{examples}

\begin{example}\label{ex:order-matters}
If $\rbar$ is peu ramifi\'ee, it is natural to ask whether $\rbar$ is
peu ramifi\'ee with respect to every (saturated) filtration on
$\rbar$. This is not the case. Suppose, for instance, that $K$ does not contain
the $p$-th roots of unity (so $\omega \ne 1$) 
and $$\rbar \cong
\begin{pmatrix}
  \omega & *_1 & *_2 \\ & 1 & 0 \\ & & 1
\end{pmatrix}$$
where the class of the cocycle $*_1$ is nontrivial and peu ramifi\'ee, and the cocycle $*_2$ is
tr\`es ramifi\'ee. For the filtration on $\rbar$ in which $\Ubar_i$ is
the span of the first $i$ vectors giving rise to the above matrix
representation (so that the action of $G_{\Qp}$ on $\Ubar_i$ is
given by the upper-left $i\times i$ block), the representation $\rbar$
is peu ramifi\'ee. This is clear at the first two steps in the
filtration, and for the third step one notes (as in Example~\ref{examples-pr}(3)) that there are no nontrivial
maps $\Ubar_2 \to \Vbar_3(1)$.

On the other hand, if one defines a new filtration on $\rbar$ by
replacing $\Ubar_2$ with the span of the first and third basis vectors
giving rise to the above matrix representation, then $\rbar$ is not
peu ramifi\'ee with respect to the new filtration, because the new
$\Ubar_2$ is tr\`es ramifi\'ee.
\end{example}

\begin{remark}
  One consequence of the preceding example is that the collection of peu ramifi\'ee
  representations is not closed under taking arbitrary
  subquotients. On the other hand, if $\rbar$ is peu ramifi\'ee with
  respect to the filtration $\{\Ubar_i\}$, then for any $a \le b$ it is
  not difficult to check that $\Ubar_b/\Ubar_a$ is peu ramifi\'ee with
  respect to the induced filtration $\{\Ubar_{a+i}/\Ubar_a\}_{0 \le i
    \le b-a}$.

 Using the preceding example one can similarly see that the collection of peu
 ramifi\'ee representations is not closed under contragredients.
\end{remark}

\begin{remark}
  In some sense we are making an arbitrary choice by demanding that we
  first lift $\Ubar_1$, then to $\Ubar_2$, then to $\Ubar_3$, and so
  forth. One could equally well lift in other orders, and as
  Example~\ref{ex:order-matters} shows, this can make a
  difference. However, since Definition~\ref{def:peu-ramif-fil} will
  suffice for our purposes, we do not elaborate further on this point.
\end{remark}

We say that a \emph{$\Zpbar$-lift} of  an $\Fpbar[G_K]$-module $\Vbar$ is a
$\Zpbar[G_K]$-module $V$ that is free as a $\Zpbar$-module,  together with a $\Fpbar[G_K]$-isomorphism $V
\otimes_{\Zpbar} \Fpbar \cong \Vbar$.
We have introduced the notion of a peu ramifi\'ee representation
(Definition~\ref{def:peu-ramif-fil})  in
order to prove the following result, to the effect that peu ramifi\'ee
representations have many $\Zpbar$-lifts.

\begin{thm}
  \label{thm: existence of lifts, prescribed on inertia, for
    peu ramifiee reps} Suppose that
  $K/\Qp$ is a finite extension. Consider a representation
  $\rbar:G_K\to\GL_n(\Fpbar)$ that is peu ramifi\'ee
  with respect to the increasing filtration $\{\Ubar_i\}$, so that $\rbar$
may be written as  \[\rbar=
  \begin{pmatrix}
    \overline{V}_1  &\dots & *\\
& \ddots &\vdots\\
&&\overline{V}_\ell
  \end{pmatrix},\]
where the $\overline{V}_i := \Ubar_i/\Ubar_{i-1}$ are the graded
pieces of the filtration.

For each~$i$, suppose that we are given a $\Zpbar$-representation
$V_i$ of~$G_K$ lifting $\Vbar_i$. 
Then there exist unramified characters $\psi_1,\dots,\psi_\ell$ with
trivial reduction such that  $\rbar$ may be lifted to a
representation~$r$ of the form \[r=
  \begin{pmatrix}
    V_1\otimes\psi_1  &\dots & *\\
& \ddots &\vdots\\
&&V_\ell \otimes\psi_\ell
  \end{pmatrix}.\]
More precisely, $r$ is equipped with an increasing filtration
$\{U_i\}$ by $\Zpbar$-direct summands such that  $U_i/U_{i-1} \cong V_i \otimes \psi_i$ and $r
\otimes_{\Zpbar} \Fpbar \cong \rbar$ induces $U_i \otimes_{\Zpbar}
\Fpbar  \cong \Ubar_i$, for each $1 \le i \le \ell$.

In fact, there are infinitely many choices
  of characters $(\psi_1,\dots,\psi_\ell)$ for which this is
  true, in the strong sense that for any $1 \le i \le \ell$, if
  $(\psi_1,\dots,\psi_{i-1})$ can be extended to an $\ell$-tuple of characters  for which such a lift
  exists, then there are infinitely many choices of $\psi_i$ such that
  $(\psi_1,\dots,\psi_i)$ can also  be extended to such an $\ell$-tuple. 
\end{thm}

\begin{proof} 
We proceed by induction on $\ell$, the case $\ell=1$ being trivial.
  From the induction hypothesis, we can find $\psi_1,\ldots,\psi_{\ell-1}$ so that
  $\Ubar := \Ubar_{\ell-1}$ can be lifted to some
 \[ U:=  \begin{pmatrix}
    V_1\otimes\psi_1  &\dots & *\\
& \ddots &\vdots\\
&&V_{\ell-1} \otimes\psi_{\ell-1}
  \end{pmatrix}.\]
as in the statement of the theorem. It suffices to prove that for each
such choice of $\psi_1,\ldots,\psi_{\ell-1}$, there
exist infinitely many choices of $\psi_{\ell}$ for which $\rbar$ lifts
to an extension of $V_{\ell} \otimes \psi_{\ell}$ by $U$ as in the
statement of the theorem.

Choose the field $E/\Qp$ large enough so that $U$ and $V_\ell$ 
are realisable over $\cO_E$, and so that $\rbar$ is
realisable over the residue field of $E$.   Suppose that $F/E$ is a
finite 
 extension with ramification degree $e(F/E) > (\dim
\Vbar_\ell)(\dim \Ubar)$,
write $\cO$ for the integers of $F$ and $\F$ for its residue field,
and let $\psi : G_K \to \cO^{\times}$ be an unramified character such that
$0 < v_E(\psi(\Frob_K)-1) < 1/(\dim
\Vbar_\ell)(\dim \Ubar)$.   In the
remainder of this argument, when we write $U$ and $V_\ell$ we will
mean \emph{their {\upshape(}chosen{\upshape)} realisations over
  $\cO$}, and similarly $\Ubar$ and $\Vbar_{\ell}$ will mean their
realisations over $\F$ obtained by reduction from $U$ and $V_\ell$.

Extensions of $V_\ell \otimes \psi$ by $U$
 correspond to elements
of~$H^1(G_K,\Hom_{\cO}(V_\ell\otimes \psi,U))$, while~$\rbar$ corresponds to an
element~$c$ of~$\Ext^1_{\F[G_K]}(\Vbar_{\ell},\Ubar)$, which we
identify with $H^1(G_K,\Hom_{\F}(\Vbar_{\ell},\Ubar))$. By hypothesis
(together with the remark about base change immediately following
Definition~\ref{def:peu-ramif-fil}) the class $c$ is annihilated by
$H^1_{\ur}(G_K,\Hom_{\F}(\Ubar,\Vbar_{\ell}(1)))$ under Tate local duality.  Taking the
cohomology of the exact sequence \[0\to
\Hom_{\cO}(V_{\ell}\otimes_{\cO}\psi,  U)\stackrel{\varpi}{\to}\Hom_{\cO}(V_{\ell}\otimes_{\cO}\psi,  U)\to\Hom_{\F}(\Vbar_{\ell},\Ubar),\]
we have in particular an exact
sequence \[H^1(G_K,\Hom_{\cO}(V_{\ell}\otimes_{\cO}\psi,  U))\to
H^1(G_K,\Hom_{\F}(\Vbar_{\ell},\Ubar))\stackrel{\delta}{\to}
H^2(G_K,\Hom_{\cO}(V_{\ell}\otimes_{\cO}\psi,  U)),\] so it is enough
to check that that~$c \in \ker(\delta)$ except for finitely many choices
of $\psi$.

From Tate duality, we have the dual map
\begin{eqnarray*}
H^0(G_K,\Hom_{\cO}(U,V_{\ell}(1) \otimes_{\cO}\psi)\otimes F/\cO) &
                                                                   \stackrel{\delta^\vee}{\to} &
H^1(G_K,\Hom_{\F}(\Ubar,\Vbar_{\ell}(1)).
\end{eqnarray*}
As $\ker(\delta)^{\perp}=\im(\delta^{\vee})$, it is enough to show
that $\im(\delta^{\vee})$ is contained in
$H^1_{\ur}(G_K,\Hom_{\F}(\Ubar,\Vbar_{\ell}(1)))$ except, again, for
possibly finitely many
choices of $\psi$.  Letting $X = \Hom_{\cO}(U,V_{\ell}(1))$,
we first claim that $(X \otimes_{\cO} \cO(\psi))^{G_K} = 0$ for all
but finitely many choices of $\psi$.  Indeed, if
 $$(X \otimes_{\cO} \cO(\psi))^{G_K} = \Hom_{\cO[G_K]}(U,V_{\ell}(1)
 \otimes_{\cO}\psi) \neq 0$$ then  we must have
$W \cong Z(1) \otimes_{\cO} \psi$ for some Jordan--H\"older factor $W$ of $U$
and $Z$ of $V_{\ell}$. This can happen for only finitely many choices
of $\psi$ (by determinant considerations applied to each of the finitely many
pairs $W,Z$). Now we are done by the following proposition.
\end{proof}

\begin{prop}
 \label{prop:unram-image}
Let $F/\Qp$ be a finite extension with ring of integers $\cO$ and
residue field $\F$.  Let $X$ be an $\cO[G_K]$-module that is free of finite rank as an
$\cO$-module. Suppose that there is a field lying $E$ lying between
$F$ and $\Qp$  such that $X$
 is realisable over $\cO_E$
 and with ramification index $e(F/E) > \rank_{\cO}(X)$. Let  $\psi : G_K \to \cO^{\times}$ be an
unramified character such that $0 < v_E(\psi(\Frob_K)-1) < 1/\rank_{\cO}(X)$.

 Assume further that
$(X\otimes_{\cO} \cO(\psi))^{G_K} = 0$.
Then the image of
$$ \delta^{\vee} : H^0(G_K, (X \otimes_{\cO} \cO(\psi)) \otimes_{\cO}
F/\cO) \to H^1(G_K, X \otimes_{\cO} \F)$$
is equal to the subspace of unramified classes, and in particular
depends only on $X \otimes_{\cO} \F$, and not on $X$, $F$, or $\psi$.
\end{prop}

\begin{proof} The statement is unchanged upon replacing $E$
  with the maximal unramified extension $E^{\ur}$ of $E$ contained in
  $F$. We are therefore
  reduced to the case where $F/E$ is totally ramified (so that in
  particular $\F$ is also the residue field of $E$).

Let $X_{\cO_E}$ be a realisation of $X$ over $\cO_E$.  Write $\Xbar = X_{\cO_E}
\otimes_{\cO_E} \F = X \otimes_{\cO} \F$ and $X_{\psi} =
  X_{\cO_E} \otimes_{\cO_E} \cO(\psi)$.
The inclusion $\iota : X_{\cO_E} \into X_\psi$ sending $x
\mapsto x \otimes 1$  is a map of
$\cO_E$-modules inducing an isomorphism of $\F[G_K]$-modules $\Xbar \cong X_\psi
\otimes_{\cO} \F$. Moreover for any $g \in G_K$ and $x \in X_{\cO_E}$ we have
$g \cdot \iota(x) = \psi(g)(\iota(g \cdot x))$, so that the map
$\iota$ is at least $I_K$-linear.

Define $\alpha = \psi(\Frob_K)^{-1} -1$ and write $N = \Frob_K - 1$, which acts on  $\Xbar^{I_K}$ with kernel
$\ker(N) = \Xbar^{G_K}$.  We have an isomorphism
$$H^1(G_K/I_K,\Xbar^{I_K}) \cong \Xbar^{I_K}/N \Xbar^{I_K}$$ induced
by evaluation at $\Frob_K$. Note that any class in this quotient space
has a representative in $\cup_{i=0}^{\infty} \ker(N^i)$, as can be seen for example
by writing $\Xbar^{I_K} = \overline{Y} \oplus \overline{Z}$ with $N$
nilpotent on $\overline{Y}$ and invertible on $\overline{Z}$. Hence to see
that the image of $\delta^{\vee}$ contains all unramified classes,
it suffices to exhibit for $\o f \in \cup_{i=0}^{\infty} \ker(N^i)$
an element $e_{\o f} \in (X_\psi \otimes_{\cO} F/\cO)^{G_K}$ such that $\delta^\vee(e_{\o f}) = [\o f]$
in $H^1(G_K/I_K,\Xbar^{I_K})$.

Suppose then that $\o f \in \cup_{i=0}^{\infty} \ker(N^i)$ is nonzero. Let $i \ge 0$
be the largest integer such that $N^i \o f \ne 0$, and let $\fbar_i := \o f$.
For each $0 \le j \le i$ let $f_j \in X_{\cO_E}$ be a lift of
$N^{i-j} \fbar_i$, and define
$$ f^* = \sum_{j=0}^i \alpha^j \cdot \iota(f_j) \in X_\psi.$$
Since $\fbar_j \in \Xbar^{I_K}$, it follows that for $g \in I_K$ we have $g(f_j)\equiv f_j \pmod{\varpi_E
  X_{\cO_E}}$ with $\varpi_E \in \cO_E$ a uniformiser, and so also $g(f^*) \equiv
f^* \pmod{\varpi_E X_\psi}$.

Now let us compute $(\Frob_K-1)(f^*)$. Noting that $(\Frob_K-1)f_j
\equiv f_{j-1} \pmod{\varpi_E X_{\cO_E}}$, with $f_{-1} := 0$, and recalling
that $(1+\alpha)(\Frob_K \cdot \iota(x)) = \iota(\Frob_K \cdot x)$, we have
\begin{eqnarray*}
  (1+\alpha)(\Frob_K(f^*)-f^*) & = & \sum_{j=0}^i \alpha^j
                                     \iota(\Frob_K(f_j)) -
                                     (1+\alpha)\sum_{j=0}^i \alpha^j
                                     \iota(f_j) \\
& = & \sum_{j=0}^i \alpha^j \iota((\Frob_K-1)f_j) - \sum_{j=0}^i
      \alpha^{j+1} \iota(f_j) \\
& \equiv & \sum_{j=0}^i \alpha^j \iota(f_{j-1}) - \sum_{j=0}^i
           \alpha^{j+1} \iota(f_j) \pmod{\varpi_E X_{\psi}} \\
& \equiv & -\alpha^{i+1} \iota(f_i) \pmod{\varpi_E X_\psi}.
\end{eqnarray*}
Note that $N^{i+1} \fbar_i= 0$ and $N^i \fbar_i \neq 0$, so that $i+1 \le
\dim_{\F} \Xbar^{I_K} \le \rank_{\cO} X$. Therefore
$v_E(\alpha^{i+1}) < 1$, and we deduce that $g(f^*) \equiv f^*
\pmod{\alpha^{i+1} X_{\psi}}$ for all $g \in G_K$, or in other words
$f^* \otimes \alpha^{-i-1} \in (X_\psi \otimes F/\cO)^{G_K}$.

Furthermore, if $c_{\fbar} := \delta^{\vee}(f^* \otimes \alpha^{-i-1}) \in
H^1(G_K,\Xbar)$, then $c_{\fbar}(g)$ is by definition the image in
$\Xbar$ of $\alpha^{-i-1}(g(f^*)-f^*)$.  So on the one hand $c_{\fbar}$ is unramified
(because $v_E(\alpha^{i+1}) < v_E(\varpi_E)=1$), while on the other
hand $c_{\fbar}(\Frob_K) = -\fbar_i$. Thus we can take $e_{\o f} := -f^* \otimes \alpha^{-i-1}$,
and we have shown that $H^1(G_K/I_K,\Xbar^{I_K}) \subset
\im \delta^{\vee}$.

On the other hand, since $X_\psi^{G_K}$ is assumed to be trivial, we
have that $(X_\psi \otimes F/\cO)^{G_K}$ is of finite length; so
if $\varpi_F$ is a uniformiser of $F$, then
\begin{eqnarray*}
  \dim(\im \delta^{\vee}) & = & \dim( (X_{\psi} \otimes
                                F/\cO)^{G_K}/\varpi_F) \\
 & = & \dim((X_{\psi} \otimes F/\cO)^{G_K}[\varpi_F]) = \dim
       \Xbar^{G_K} = \dim (\ker N).
\end{eqnarray*}
On the other hand $\dim (\ker N) = \dim (\coker N) = \dim
H^1(G_K/I_K,\Xbar^{I_K})$, and the result follows.
\end{proof}

Theorem~\ref{thm: existence of lifts, prescribed on inertia, for
    peu ramifiee reps} implies the following result on the existence
  of certain potentially crystalline Galois representations.

\begin{prop}
  \label{prop: existence of lifts, prescribed on inertia, for
    peu ram}Suppose that
  $K/\Qp$ is a finite extension. Consider a representation
  $\rbar:G_K\to\GL_n(\Fpbar)$ that  is peu ramifi\'ee
  with respect to the increasing filtration $\{\Ubar_i\}$, so that $\rbar$
may be written as  \[\rbar=
  \begin{pmatrix}
    \overline{V}_1  &\dots & *\\
& \ddots &\vdots\\
&&\overline{V}_\ell
  \end{pmatrix},\]
where the $\overline{V}_i := \Ubar_i/\Ubar_{i-1}$ are the graded
pieces of the filtration.

For each~$i$, suppose that we are given a $\Zpbar$-representation
$V_i$ of~$G_K$ lifting $\Vbar_i$ such that:
\begin{itemize}
\item each $V_i$ is potentially crystalline, and
\item for each $1 \le i < \ell$ and each $\kappa : K \into \Qpbar$,
  every element of~$\HT_{\kappa}(V_{i+1})$ is strictly greater than every
  element of~$\HT_{\kappa}(V_{i})$.
\end{itemize}
Then $\rbar$ may be lifted to a potentially crystalline
representation~$r$ of the form \[r=
  \begin{pmatrix}
    V_1\otimes\psi_1  &\dots & *\\
& \ddots &\vdots\\
&&V_\ell \otimes\psi_\ell
  \end{pmatrix},\]where each $\psi_i$ is an unramified character
  with trivial reduction, and if $K'/K$ is a finite extension such that each
  $V_i|_{G_{K'}}$ is crystalline, then $r|_{G_{K'}}$ is also crystalline. 

In fact, there are infinitely many choices
  of characters $(\psi_1,\dots,\psi_\ell)$ for which this is
  true, in the strong sense that for any $1 \le i \le \ell$, if
  $(\psi_1,\dots,\psi_{i-1})$ can be extended to an $\ell$-tuple of characters  for which such a lift
  exists, then there are infinitely many choices of $\psi_i$ such that
  $(\psi_1,\dots,\psi_i)$ can also  be extended to such an $\ell$-tuple.
\end{prop}
\begin{proof}

This follows from Theorem~\ref{thm: existence of lifts, prescribed on inertia, for
    peu ramifiee reps} 
  along with
  standard facts about
  extensions of de Rham representations. Indeed, by
  \cite[Prop.~1.28(2)]{MR1263527} and
our assumption on the
Hodge--Tate weights of the~$V_i$, 
the representation $r|_{G_{K'}}$ is
  semistable for any $r$ as in Theorem~\ref{thm: existence of lifts, prescribed on inertia, for
    peu ramifiee reps} and any $K'$ as above. Then by repeated
  application of the third
  part of \cite[Prop.~1.24(2)]{MR1263527}, as well as \cite[Prop.~1.26]{MR1263527}, this semistable
  representation is guaranteed to be crystalline as long as there is
  no $G_{K'}$-equivariant surjection $(V_j \otimes V_i^{*})(\psi_j \psi_i^{-1})
  \twoheadrightarrow \cyclo$ for any $j < i$. Once
  $\psi_1,\dots,\psi_{i-1}$ have been determined, this can be arranged
  by avoiding finitely many possibilities for $\psi_i$. 
\end{proof}

We give two sample applications of Proposition~\ref{prop: existence of lifts, prescribed on inertia, for
    peu ram}.
The following Corollary will be used in forthcoming work of Arias de
Reyna and Dieulefait (in the special case where $\rbar$ is
Fontaine--Laffaille and the Hodge type $\lambda$ is $0$).
\begin{cor}
\label{cor: peu ramimplies HT type whatever and bounded tame type} Fix an integer $n\ge 1$. Then there is a
finite extension $K'/K$, depending only on $n$, with the following
property: if $\rbar:G_K\to\GL_n(\Fpbar)$ is peu ramifi\'ee and
$\lambda=(\lambda_{\kappa,i})\in(\Z^n_+)^{\Hom_\Qp(K,\Qpbar)}$, then
$\rbar$ has a potentially
diagonalisable lift $r:G_K\to\GL_n(\Zpbar)$ that is regular of Hodge type~$\lambda$,
with the property that $r|_{G_{K'}}$ is crystalline.
\end{cor}

\begin{proof}
Write $\rbar$ as in Proposition~\ref{prop: existence of lifts, prescribed on inertia, for
    peu ram} with
$\Vbar_i$ irreducible for all $i$, and set $d_i = \dim_{\Fpbar} \Ubar_i$.
By Proposition~\ref{prop: existence of lifts, prescribed on inertia, for
    peu ram} and~\cite[Lem.\ 1.4.3]{BLGGT}, it is enough to show that there is a finite extension $K'/K$
  depending only on~$n$, with the property that  we may lift each~$\Vbar_i$ to a
  potentially crystalline representation~$V_i$, such that for all
  $i,\kappa$ the set $\HT_{\kappa}(V_i)$ is equal to $\{
  \lambda_{\kappa,n-j} + j \, : \, j \in [d_{i-1},d_i-1]\},$
  with the additional property that $V_i|_{G_{K'}}$ is isomorphic to a
  direct sum of crystalline characters. This
  is immediate from Lemma~\ref{lem:existence of a bounded lift for an irreducible rep over
    an unramified field} below.
\end{proof}
\begin{lem}
  \label{lem:existence of a bounded lift for an irreducible rep over
    an unramified field} Let $d\ge 1$ be an
  integer. Let $K_d$ be the unramified extension of $K$ of degree $d$,
  and define $L$ to be any totally ramified extension of $K_d$ of
  degree $|k_d^{\times}|$, where $k_d$ is the residue field of $K_d$.
  Let
  $\rbar:G_K\to\GL_d(\Fpbar)$  be an irreducible representation. Then
  for any collection of multisets of $d$ integers
  $\{h_{\kappa,1},\dots,h_{\kappa,d}\}$, one for each continuous embedding
  $\kappa : K \into \Qpbar$, there is a lift of $\rbar$ to a representation
  $r:G_{K}\to\GL_d(\Zpbar)$, such that $r|_{G_{L}}$ is isomorphic to
  a direct sum of crystalline characters,
  and for each $\kappa$ we have $\HT_{\kappa}(r) = \{h_{\kappa,1},\dots,h_{\kappa,d}\}$.
\end{lem}
\begin{proof}
  Since $\rbar$ is irreducible, we can write
  $\rbar\cong\Ind_{G_{K_d}}^{G_K}\psibar$, and~$\psibar:G_{K_d} \to\Fpbar^\times$ is a
  character. Choose a crystalline character
  $\chi:G_{K_d} \to\Qpbar^\times$ with the property that for each
  continuous embedding
  $\kappa:K\into\Qpbar$ we have  $$\bigcup_{\tilde{\kappa}|_K = \kappa}
  \HT_{\tilde{\kappa}}(\chi) = \{h_{\kappa,1},\dots,h_{\kappa,d}\},$$
  where the union is taken as multisets. (That such a
  character exists is well-known; see e.g.\ \cite[\S2.3,
  Cor.~2]{MR563476}.)
If we let $\theta:G_{K_d} \to\Zpbar^\times$ be the
  Teichm\"uller lift of $\psibar\chibar^{-1}$, then we may take
  $r:=\Ind_{G_{K_d}}^{G_K}\chi\theta$, which has the correct
  Hodge--Tate weights by \cite[Cor.~7.1.3]{EGHS}. (Note that
  ${}^g \theta |_{G_{L}}$ is unramified for any $g \in G_K$.)
\end{proof}

As a second application of Proposition~\ref{prop: existence of lifts, prescribed on inertia, for
    peu ram}, we show that each peu ramifi\'e representation has a
  crystalline lift of some Serre weight.

  \begin{cor}\label{cor:crys-lift-of-serre-weight}
    Suppose that $K/\Qp$ is a finite extension, and that $\rbar: G_K
    \to \GL_n(\Fpbar)$ is peu ramifi\'ee. Then $\rbar$ has a
    crystalline lift of some Serre weight.
  \end{cor}

  \begin{proof}
    When $\rbar$ is irreducible, this is straightforward from
    \cite[Thm~B.1.1]{EGHS}. (One only has to note that when $\rbar$ is
    irreducible, an obvious lift of $\rbar$ in the terminology of
    \cite[\S 7]{EGHS} is always an unramified twist of a true lift of
    $\rbar$.)

In the general case, suppose that $\rbar$ is peu
    ramifi\'ee with respect to the filtration  $\{\Ubar_i\}$, and as
    usual set $\Vbar_i := \Ubar_i/\Ubar_{i-1}$. By the previous
    paragraph, for each $\Vbar_i$ we are able to choose a crystalline lift
    $V_i$ of some Serre weight. By an argument as in the fourth
    paragraph of the proof of \cite[Thm~B.1.1]{EGHS} it is possible to
    arrange that every element of $\HT_{\kappa}(V_{i+1})$ is strictly greater
    than every element of $\HT_{\kappa}(V_i)$, and that $\oplus_i V_i$
    is a crystalline lift of $\oplus_i \Vbar_i$ of some Serre weight. 
 (This is just a matter of
    replacing each $V_i$ with a twist by a suitably-chosen crystalline
    character of trivial reduction.) Now the Corollary follows directly from Proposition~\ref{prop: existence of lifts, prescribed on inertia, for
    peu ram} (with $K' = K$).
\end{proof}

\subsection{Highly twisted lifts}\label{subsec:
  universally twisted lifts}

In this section we give a criterion (Proposition~\ref{prop:pr-htl}) for checking that a representation
is peu ramifi\'ee, which we will apply to show in Section~\ref{subsec: FL
  representations are universally twisted} that Fontaine--Laffaille
representations are peu ramifi\'ee.

\begin{defn}
  \label{defn: universally twisted lift} Suppose that $K/\Qp$ is a
finite extension.  Consider a representation
  $\rbar:G_K\to\GL_n(\Fpbar)$, let $\Vbar$ be the underlying
  $\Fpbar[G_K]$-module of $\rbar$, and suppose that $0 = \Ubar_0
  \subset \Ubar_1 \subset \dots \subset \Ubar_{\ell} = \Vbar$ is an
  increasing filtration on $\Vbar$ by $\Fpbar[G_K]$-submodules. Denote $\Vbar_i :=
  \Ubar_i/\Ubar_{i-1}$ for $1 \le i \le \ell$, the graded pieces of the filtration.

We say that \emph{$\rbar$ admits highly twisted lifts with
  respect to the filtration  $\{\Ubar_i\}$} if
there exist $\Zpbar$-lifts $V_i$ of the $\Vbar_i$, and a family
of $\Zpbar$-lifts $V(\psi_1,\ldots,\psi_\ell)$ of $\Vbar$ indexed by a nonempty
set $\Psi$ of
$\ell$-tuples of unramified characters
$\psi_i:G_K\to\Zpbartimes$
with trivial reduction modulo~$\m_{\Zpbar}$, having the
following additional properties:
\begin{itemize}
\item Each $V(\psi_1,\ldots,\psi_\ell)$ is equipped with an increasing filtration
  $\{U(\psi_1,\ldots,\psi_\ell)_i\}$ by $\Zpbar[G_K]$-submodules that are $\Zpbar$-direct summands.
\item  We have $U(\psi_1,\ldots,\psi_\ell)_i/U(\psi_1,\ldots,\psi_\ell)_{i-1}
  \cong  V_i \otimes \psi_i$ for each $1 \le i \le \ell$.
\item The isomorphism $V(\psi_1,\ldots,\psi_\ell) \otimes_{\Zpbar}
  \Fpbar \cong  \Vbar$ induces isomorphisms $U(\psi_1,\ldots,\psi_\ell)_i \otimes_{\Zpbar}
  \Fpbar \cong \Ubar_i$ for each $1 \le i \le \ell$.
\item $U(\psi_1,\ldots,\psi_{\ell})_i$ depends up to isomorphism only
  on $\psi_1,\ldots,\psi_i$ (that is, it does not depend on
  $\psi_{i+1},\ldots,\psi_\ell$).

\item For each $(\psi_1,\ldots,\psi_i)$ that extends to an element of
  $\Psi$ and for each $\epsilon > 0$,  there exists $\psi_{i+1}$ such
  that $(\psi_1,\ldots,\psi_{i+1})$ extends to an element of $\Psi$,
  with the further property that $0 < v_{\Qp}(\psi_{i+1}(\Frob_K)-1) < \epsilon$.
 \end{itemize}

If moreover the set $\Psi$ can be taken to be the set of all $\ell$-tuples of
unramified characters $\psi_i : G_K \to \Zpbar^{\times}$ with trivial
reduction modulo $\mathfrak{m}_{\Zpbar}$, we say that $\rbar$
\emph{admits universally twisted lifts with respect to the
  filtration $\{\Ubar_i\}$.}
\end{defn}

As with Definition~\ref{def:peu-ramif-fil}, the preceding definition is
most interesting in the case where the filtration
$\{\Ubar_i\}$ is saturated, and so we make the following further definition.

\begin{defn}
  \label{defn:utl-saturated}
We say that \emph{$\rbar$ admits highly {\upshape(}resp.\
  universally{\upshape)} twisted lifts} if it
admits highly (resp.\ universally) twisted lifts as in Definition~\ref{defn: universally twisted lift} with
  respect to some saturated filtration.
\end{defn}

\begin{rem}
  \label{rem:any filtration will do?} 
It is natural to ask whether, if $\rbar$
admits highly (resp.\ universally) twisted lifts with respect to some saturated filtration as
in Definition~\ref{defn:utl-saturated}, it admits
highly (resp.\ universally) twisted lifts with respect to any such
filtration. Proposition~\ref{prop:pr-htl} below, in combination with
Example~\ref{ex:order-matters}, gives a negative answer to this
question in the highly twisted case. 

In fact, Example \ref{ex:order-matters} also shows that the above question has a negative answer in the universally twisted
case.  Suppose for simplicity that $K/\Qp$ is unramified and that $p > 2$. Then $\rbar$ in Example \ref{ex:order-matters}
admits universally twisted lifts for the first filtration considered there. To see this, we first note that the first block $ \Ubar_2
= \begin{pmatrix}\omega & *_1 \\ & 1 \end{pmatrix}$ admits universally twisted lifts for $V_1 = \varepsilon$ and $V_2 =1$
by Proposition \ref{prop:fllift} below, because $ \Ubar_2$ is Fontaine--Laffaille. Since there is no nontrivial map $ \Ubar_2 \to \Vbar
_3 (1)$, one easily checks that $\bar r$ admits universally twisted lifts for this filtration. However, $\rbar$ does not admit
universally twisted lifts for the second filtration considered in Example \ref{ex:order-matters}. This is because
the first block $ \Ubar'_2 = \begin{pmatrix}\omega & *_2 \\ & 1 \end{pmatrix}$ does not admit universally twisted lifts
(e.g.\ by Proposition~\ref{prop:pr-htl}).
\end{rem}

\begin{prop}\label{prop:pr-htl}
  Let $K/\Qp$ be a finite extension, and let $\{\Ubar_i\}$ be an
  increasing filtration on the representation $\rbar : G_K \to
  \GL_n(\Fpbar)$. Then $\rbar$ is peu ramifi\'ee with respect to
  $\{\Ubar_i\}$ if and only if it admits highly twisted lifts with respect to $\{\Ubar_i\}$.
\end{prop}

\begin{proof}
An inspection of the proof of Theorem~\ref{thm: existence of lifts, prescribed on inertia, for
    peu ramifiee reps} already gives the ``only if'' implication (for
  \emph{any} choice of $V_i$'s lifting $\Vbar_i$).

For the other direction, we assume that $\rbar$ admits highly
twisted lifts with respect to the filtration $\{\Ubar_i\}$ and some $\Zpbar$-lifts $V_i$ of the $\Vbar_i$.
We proceed by induction on $\ell$, the length
of the filtration. By the induction hypothesis we may assume that for
all $i < \ell$ the class in
  $H^1(G_K,\Hom_{\F}(\Vbar_i,\Ubar_{i-1}))$  defined by $\Ubar_i$ is annihilated under Tate local
  duality by
  $H^1_{\ur}(G_K,\Hom_{\F}(\Ubar_{i-1},\Vbar_i(1)))$, and it remains
  to prove this for $i = \ell$. Choose any
  $(\psi_1,\ldots,\psi_{\ell-1})$ that extends to an element of the set $\Psi$
  (as in Definition~\ref{defn: universally twisted lift} for $\rbar$), and let
  $U := U(\psi_1,\dots,\psi_{\ell})_{\ell-1}$ (where $\psi_{\ell}$ is any
  character such that $(\psi_1,\ldots,\psi_\ell) \in \Psi$); note that
  this is independent of $\psi_{\ell}$.

Let
  $\mathcal{S}$ be the set of characters $\psi : G_K \to
  \Zpbar^{\times}$ such that $(\Hom_{\Zpbar}(U,V_{\ell}(1)) \otimes_{\Zpbar} \Zpbar(\psi))^{G_K} \neq
  0$.  As in the proof of Theorem~\ref{thm: existence of lifts, prescribed on inertia, for
    peu ramifiee reps} we see that $\mathcal{S}$ is finite. Let $E/\Qp$ be a finite extension such
  that $U$ and $V_{\ell}$ are realisable over $\cO_E$.   It
  follows from the
  highly twisted lift condition on $\rbar$ that there exists
  $\psi_\ell$ having the following properties:
  \begin{enumerate}
  \item[(i)]  $(\psi_1,\ldots,\psi_\ell) \in \Psi$,
 \item[(ii)] $\psi_\ell \not \in \mathcal{S}$, and
\item[(iii)] $0 < v_E(\psi_\ell(\Frob_K)-1) < 1/(\dim \Ubar_{\ell-1})(\dim \Vbar_{\ell})$.
  \end{enumerate}

Let $F/E$ be a finite extension over which $\psi_\ell$ and
$V(\psi_1,\ldots,\psi_{\ell})$ are both realisable.  Write $\cO$ for the ring of integers of $F$, and $\F$ for its
residue field.   For the remainder of this proof, when we
  write $U$, $V_{\ell}$, $\psi_{\ell}$  we will mean their
  chosen realisations over $F$, and similarly for
$\Ubar$, $\Vbar_{\ell}$ over $\F$ (obtained by reduction).

Set $X =
  \Hom_{\cO}(U,V_{\ell}(1))$.
 As in the proof of Theorem~\ref{thm: existence of lifts, prescribed on inertia, for
    peu ramifiee reps},  write $\delta$ for the connection map
\[ H^1(G_K,\Hom_{\F}(\Vbar_{\ell},\Ubar))\stackrel{\delta}{\to}
H^2(G_K,\Hom_{\cO}(V_{\ell}\otimes_{\cO}\psi_\ell,  U)).\]
The existence of the lift
$V(\psi_1,\ldots,\psi_{\ell})$ (i.e.\ the property (i) of $\psi_\ell$) shows that the class
 $c \in H^1(G_K,\Hom_{\F}(\Vbar_\ell,\Ubar))$ defining $\rbar$ lies in
 $\ker(\delta)$. On the other hand, the properties (ii) and (iii) of
 $\psi_\ell$ mean that Proposition~\ref{prop:unram-image} applies (with
$\psi_\ell$ playing the role of $\psi$) to show that the dual map
$\delta^\vee$ has image
$H^1_{\ur}(G_K,X\otimes_{\cO} \F)$.  Since $c\in \ker(\delta)$ it is annihilated under
Tate local duality by this image, and we deduce that $\rbar$ is peu ramifi\'ee.
\end{proof}

  \begin{cor}
Suppose that $\rbar$ admits highly twisted lifts with respect to
the filtration $\{\Ubar_i\}$. Then $\rbar$ satisfies the definition of
admitting highly twisted lifts with respect to
the filtration $\{\Ubar_i\}$ for \emph{any} lifts $V_i$ of
the $\Vbar_i$.
  \end{cor}

  \begin{proof}
    This is immediate from Proposition~\ref{prop:pr-htl} along with     the first sentence of its proof.
  \end{proof}
\begin{remark} 
The above corollary fails if we replace `highly twisted' with
`universally twisted'. For instance, consider Example
\ref{examples-pr}(1) with $K/\Qp$ unramified, the extension class $*$ peu ramifi\'ee, and
$\chi =1$. It admits universally twisted lifts if we set $V_1 =
\varepsilon$ and $V_2 =1$. (This will follow from Proposition~\ref{prop:fllift} below.) But it does not admit universally twisted lifts for $V_1 = \varepsilon^p$ and $V_2=1$.
\end{remark}
\begin{remark}
We do not know whether there exist representations that admit highly twisted lifts but not universally twisted lifts.
  \end{remark}

\subsection{Fontaine--Laffaille representations}\label{subsec: FL
  representations are universally twisted}In this section we will
prove that representations which admit a Fontaine--Laffaille
lift also admit universally twisted lifts, and so by Proposition~\ref{prop:pr-htl} are
peu ramifi\'ee. We begin by
recalling the formulation of unipotent Fontaine--Laffaille theory
in~\cite[\S1.1.2]{MR2103471}. Throughout this section let $K/\Qp$ be a finite
unramified extension with integer ring $\cO_K$, and write $\Frob_p$
for the absolute geometric Frobenius on~$K$.

Let~$\cO$ be the ring of integers in~$E$, a
finite extension of~$\Qp$ with residue field~$\F$. We assume that $E$
is sufficiently large as to contain the image of some (hence any) continuous embedding of $K$ into an algebraic closure of $E$.
Fix an integer $0 \le h \le p-1$, and let $\mc{MF}^h_{\cO}$
denote the category of finitely generated $\mc{O}_K \otimes_{\Zp}
\cO$-modules $M$ together with
\begin{itemize}
\item a decreasing filtration $\Fil^s M$ by $\mc{O}_K
  \otimes_{\Zp} \cO$-submodules which are $\mc{O}_K$-direct
  summands with $\Fil^0 M = M$ and $\Fil^{h+1}M=\{0\}$;
\item and $\Frob_p^{-1} \otimes 1$-linear maps $\Phi^s: \Fil^s M \rightarrow
  M$ with $\Phi^s|_{\Fil^{s+1}M} = p \Phi^{s+1}$ and $\sum_s \Phi^s
  (\Fil^s M) = M$.
\end{itemize}
We say that an object $M$ of $\mc{MF}^{p-1}_{\cO}$ is \emph{\'etale} if
$\Fil^{p-1} M = M$, and define $\mc{MF}^{p-1,u}_{\cO}$ to be the full
subcategory of $\mc{MF}^{p-1}_{\cO}$ consisting of objects with no nonzero
\'etale quotients. 
Such objects are said to be \emph{unipotent}.
Note that $\mc{MF}_{\cO}^{p-2}$ is a subcategory of $\mc{MF}_{\cO}^{p-1,u}$.

In the following paragraphs, let $\mc{MF}_{\cO}$ denote either
$\mc{MF}_{\cO}^h$ (for $0 \le h \le p-2$) or $\mc{MF}_{\cO}^{p-1,u}$
(for $h=p-1$). 
Let $\Rep_{\cO}(G_K)$ denote the category of finitely generated
$\cO$-modules with a continuous $G_K$-action.  There is an exact,
fully faithful, covariant functor of $\cO$-linear categories
$T_K : \mc{MF}_{\cO} \rightarrow \Rep_{\cO}(G_K)$. This is the functor
denoted $\mathbb{V}$ in \cite[\S1.1.2]{MR2103471}. 
The essential image of $T_K$ is closed under taking
subquotients. 
If $M$ is an object of $\mc{MF}_{\cO}$,
then the length of $M$ as an $\cO$-module is $[K : \Qp]$
times the length of $T_K(M)$ as an $\cO$-module.

Let $\mc{MF}_{\F}$ denote the full subcategory of $\mc{MF}_\cO$
consisting of objects killed by the maximal ideal of $\cO$ and let
$\Rep_{\F}(G_K)$ denote the category of finite $\F$-modules with a
continuous $G_K$-action. Then $T_K$ restricts to a functor
$\mc{MF}_{\F} \rightarrow \Rep_{\F}(G_K)$.  If $M$ is an object of
$\mc{MF}_{\F}$ and $\kappa$ is a continuous embedding $K \into \Qbar_p$,
we let $\FL_\kappa(M)$ denote the multiset of integers $i$ such that
$ \gr^i M\otimes_{\cO_K\otimes_{\Zp}\cO, \kappa\otimes 1}\cO \neq \{0\}$
and $i$ is counted with multiplicity equal to the $\F$-dimension of
this space. If $M$ is a $p$-torsion free object of $\mc{MF}_{\cO}$
then $T_K(M) \otimes_{\Z_p}\Q_p$ is crystalline and for every
continuous embedding $\kappa:K \into \Qbar_p$ we have
\[ \HT_\kappa(T_K(M) \otimes_{\Z_p}\Q_p) = \FL_\kappa(M
\otimes_{\cO} \F). \]

Moreover, if $\Lambda$ is a $G_K$-invariant lattice in a crystalline representation $V$ of $G_K$
with all its Hodge--Tate numbers in the range $[0,h]$, having (when
$h=p-1$) no nontrivial
quotient isomorphic to a twist of an unramified
representation by $\cyclo^{-(p-1)}$, then $\Lambda$
is in the essential image of $T_K$.  If some twist of $\rbar:G_K\to\GL_n(\F)$ lies in the essential image
of $T_K$ on $\mc{MF}_{\cO}^{p-2}$, we say that $\rbar$ \emph{admits  a
Fontaine--Laffaille lift}, while if some twist of  $\rbar$ lies in the essential image of
$T_K$ on $\mc{MF}_{\cO}^{p-1,u}$ we say that it  \emph{admits a unipotent extended Fontaine--Laffaille
lift}. 

The proof of the following result is essentially the same as that of
\cite[Lem.~1.4.2]{BLGGT}. (We remark that \cite[\S1.4]{BLGGT} uses the
formulation of Fontaine--Laffaille theory as \cite[\S2.4.1]{cht},
which in fact is equivalent to that of \cite[\S1.1.2]{MR2103471} (at
least on $\mc{MF}^{p-2}_{\cO}$), although this equivalence is not
needed for the following argument.)

\begin{prop}
  \label{prop:fllift} Let $K/\Qp$ be
  unramified.  Consider a representation
  $\rbar:G_K\to\GL_n(\Fpbar)$ with an increasing filtration
  $\{\Ubar_i\}$ such that $\Ubar_0=0$ and $\Ubar_{\ell} = \rbar$, so that
$\rbar$ may be written as  \[\rbar=
  \begin{pmatrix}
    \overline{V}_1  &\dots & *\\
& \ddots &\vdots\\
&&\overline{V}_{\ell}
  \end{pmatrix},\]where the $\overline{V}_i = \Ubar_{i}/\Ubar_{i-1}$
  are the graded pieces of the filtration.

Suppose that~$\rbar$ admits a Fontaine--Laffaille \emph{(}resp.\
unipotent extended Fontaine--Laffaille\emph{)} lift.
 Then $\rbar$ admits universally twisted lifts with
  respect to the filtration $\{\Ubar_i\}$; indeed, it admits universally twisted
  Fontaine--Laffaille \emph{(}resp.\ unipotent extended
  Fontaine--Laffaille\emph{)} lifts. In either case $\rbar$ is peu
  ramifi\'ee. 
  \end{prop}

  \begin{rem}
  \label{rem:nilpotent-FL}
  By duality, the same result holds when $\rbar$ admits a
  \emph{nilpotent} extended Fontaine--Laffaille lift, i.e.,
  if some twist of $\rbar$ lies in the essential image of $T_K$ on $\mc{MF}_{\cO}^{p-1,n}$,
  the full subcategory of $\mc{MF}_{\cO}^{p-1}$ whose objects admit no nonzero subobject $M$ with
  $\Fil^{1} M = 0$. We refer the reader to \cite{GaoLiu12} for a
  further discussion of nilpotent Fontaine--Laffaille theory.

Similar arguments (which we omit to keep the paper at a reasonable
length) can be used to show that the same result holds when 
$\rbar$ is finite
flat (for arbitrary $K/\Qp$; in this case the argument uses Kisin
modules).
  \end{rem}

  \begin{proof}[Proof of Proposition~\ref{prop:fllift}] Since the
    truth of this proposition for $\rbar$ evidently implies its truth
    for any twist of $\rbar$ (using the fact that every character of $G_K$ admits
    a crystalline lift), we reduce to the case that $\rbar$ lies in
    the essential image of $T_K$ on $\mc{MF}_{\cO}^{p-2}$ (or on
    $\mc{MF}_{\cO}^{p-1,u}$, in the unipotent extended case).

  The case that each~$\overline{V}_i$ is one-dimensional
    is essentially found in~\cite[Lem.\ 1.4.2]{BLGGT} and, as
    previously remarked, we will follow the
    proof of that result closely. We can and do suppose that $\rbar$
    is defined over some finite field~$\F$, and we fix a finite
    extension~$E$ of~$\Qp$ with ring of integers~$\cO$ and residue
    field~$\F$. 

Let $\overline{V}$ be the underlying $\F$-vector space of~$\rbar$, and
let $\barM$ denote the object of $\mc{MF}_{\F}$ corresponding
to~$\Vbar$, which exists by our assumption that~$\rbar$ has a
(possibly unipotent extended) Fontaine--Laffaille lift. Then we have a filtration
  \[ \barM=\barM_\ell \supset \barM_{\ell-1} \supset \dots \supset \barM_{1} \supset \barM_0=(0) \]
  by $\mc{MF}_{\F}$-subobjects such that $\barM_{i}$ corresponds to
    $\Ubar_i$ and  so $\barM_{i}/\barM_{i-1}$
  corresponds to~$\Vbar_i$. Then we claim that we can find an object $M$ of $\mc{MF}_{\cO}$ which
  is $p$-torsion free together with a filtration by $\mc{MF}_{\cO}$-subobjects
  \[ M=M_\ell \supset M_{\ell-1} \supset \dots \supset M_{1} \supset M_0=(0) \]
  and an isomorphism
  \[ M \otimes_{\cO} \F \cong \barM \]
  under which $M_i \otimes_{\cO} \F$ maps isomorphically to $\barM_i$ for all $i$.

Write $d_i:=\dim \Vbar_i$. We note first that $\barM$ has an
$\F$-basis $\bare_{i,\kappa}$ for $i=1,\dots,n$ and $\kappa \in \Hom_{\Qp}(K,
\Qpbar)$ such that
  \begin{itemize}
  \item the residue field $k$ of $K$ acts on $\bare_{i,\kappa}$ via $\kappa$;
  \item $\barM_j$ is spanned over $\F$ by the $\bare_{i,\kappa}$ for $i
    \le d_1+\dots+d_j$;
  \item  and for each $j,s$ there is a subset $\Omega_{j,s} \subset \{1,\dots,n\} \times \Hom_{\Qp}(K,\Qbar_p)$ such that
  $\barM_j \cap \Fil^s \barM$ is spanned over $\F$ by the $\bare_{i,\kappa}$ for $(i,\kappa) \in \Omega_{j,s}$.
  \end{itemize}
(Such a basis is easily constructed recursively in $j$. The
case~$j=1$ is trivial, and it is straightforward to extend a basis of
this kind for~$\barM_{j-1}$ to one for~$\barM_j$.)  We put $\Omega_s := \Omega_{\ell,s}$.

  Then we define $M$ to be the free $\cO$-module with basis $e_{i,\kappa}$ for $i=1,\dots,n$ and $\kappa \in \Hom_{\Qp}(K, \Qbar_p)$.
  \begin{itemize}
  \item We let $\cO_K$ act on $e_{i,\kappa}$ via $\kappa$;
  \item we define $M_j$ to be
  the  $\cO$-submodule generated by the  $e_{i,\kappa}$ with $i \le d_1+\dots+d_j$;
  \item and
  we define $\Fil^s M$ to be the $\cO$-submodule spanned by the $e_{i,\kappa}$ for $(i,\kappa) \in \Omega_s$.
  \end{itemize}
  We define $\Phi^s:\Fil^s M \to M$ by reverse induction on $s$. If we have defined $\Phi^{s+1}$
  we define $\Phi^s$ as follows:
  \begin{itemize}
  \item If $(i,\kappa) \in \Omega_{s+1}$ then $\Phi^s e_{i,\kappa}=p\Phi^{s+1}e_{i,\kappa}$.
  \item If $(i,\kappa) \in \Omega_s-\Omega_{s+1}$ then $\Phi^s
    e_{i,\kappa}$ is chosen to be any lift of $\barPhi^s\bare_{i,\kappa}$
    in $\sum_{i' \le d_1 + \dots + d_j} \cO \cdot e_{i',\kappa \circ
      \Frob_p}$, where $j$ is minimal such that $i
    \le d_1 + \dots + d_j$.
  \end{itemize}
  It follows from Nakayama's lemma that $M$ is an object of
  $\mc{MF}^{h}_{\cO}$.  When $h=p-1$, suppose that $M \to M'$ is
  a nontrivial \'etale quotient of $M$. We can without loss of
  generality replace $M'$ with $M'
  \otimes_{\cO} \F$; but then the map $M \to M'$ would factor through
  $\barM$, contradicting the assumption that $\barM$ is an object of
  $\mc{MF}_{\cO}$ (and not just $\mc{MF}_{\cO}^{p-1}$). It follows that
  $M$ is also an object of $\mc{MF}_{\cO}$. In the same way we see that
  $\{M_i\}$ is an increasing filtration of subobjects of $M$ in $\mc{MF}_{\cO}$.

It is immediate that
  $M$ verifies the desired property that $M_i
  \otimes_{\cO} \F$ maps isomorphically to $\Mbar_i$ under the
  isomorphism $M \otimes_{\cO} \F \cong \Mbar$.

Set $V_i:= T_K(M_{i}/M_{i-1}) \otimes_{\cO} \Zpbar$. We claim that for this choice
of~$V_i$, the conditions of Definition~\ref{defn: universally
  twisted lift} are satisfied. Since Fontaine--Laffaille theory is
compatible in an obvious fashion with extension of scalars from~$E$ to
a finite extension of~$E$, we can and do suppose that the
characters~$\psi_i$ are valued in~$\cO^\times$. Then the objects
of~$\mc{MF}_{\cO}$ corresponding to the desired lifts $V(\psi_1,\ldots,\psi_\ell)$
are obtained from $M$ by rescaling the maps~$\Phi^s$. More precisely,
if we
let $M(\psi_1,\dots,\psi_\ell)$ be defined from~$M$ by  rescaling $\Phi^s e_{i,\kappa}$ by
$\psi_j(\Frob_K)$ for  $(i,\kappa) \in
\Omega_{j,s} \setminus \Omega_{j-1,s}$,
then one can take $V(\psi_1,\dots,\psi_\ell) =
T_K(M(\psi_1,\dots,\psi_\ell)) \otimes_{\cO} \Zpbar$. 
(To establish the second bullet point in Definition~\ref{defn:
  universally twisted lift}, note from \cite[p.~670]{MR2103471} that
$T_K$ is compatible with tensor products, and use (1) of \emph{loc.\
  cit.}\ to compute the Fontaine--Laffaille module corresponding to
each $\psi_i$.)
  \end{proof}

\begin{cor}
  \label{cor: existence of lifts, prescribed on inertia, for
    FL} Suppose that $\rbar$ admits a {\upshape(}possibly unipotent extended{\upshape)}
  Fontaine--Laffaille lift. Then the conclusions of Theorem~\ref{thm: existence of lifts, prescribed on inertia, for
    peu ramifiee reps} and Proposition~\ref{prop: existence of lifts, prescribed on inertia, for
    peu ram} hold for $\rbar$ with respect to any separated,
  exhaustive increasing filtration
  $\{\Ubar_i\}$ on $\rbar$.
\end{cor}

\begin{cor}\label{cor:ld-sadr}
  Suppose that $\rbar$ admits a {\upshape(}possibly unipotent extended{\upshape)}
  Fontaine--Laffaille lift. Then the conclusion of Corollary~\ref{cor: peu ramimplies HT type whatever
    and bounded tame type} holds for $\rbar$.
\end{cor}

The following result will be used in~\cite{EGHS}. 
\begin{cor}
  \label{cor: FH FL niveau 1 lemma}Suppose that
  $\rbar:G_{\Qp}\to\GL_n(\Fpbar)$ admits a
  \emph{(}possibly unipotent extended\emph{)}
  Fontaine--Laffaille lift. 
  Suppose also
  that \[\rbar=
  \begin{pmatrix}
    \chibar_1  &\dots & *\\
& \ddots &\vdots\\
&&\chibar_n
  \end{pmatrix}.\] Suppose that $h_1>\dots>h_n$ are integers such that
  $\chibar_i|_{I_{\Qp}}=\omega^{h_i}$. Then $\rbar$ has a crystalline
  lift of the form \[r=
  \begin{pmatrix}
    \chi_1  &\dots & *\\
& \ddots &\vdots\\
&&\chi_n
  \end{pmatrix},\]where $\chi_i|_{I_{\Qp}}=\cyclo^{h_i}$.
\end{cor}
\begin{proof}
This is immediate from Corollary~\ref{cor: existence of lifts, prescribed on inertia, for
    FL}, taking the $V_i$ to be appropriate unramified twists of $\cyclo^{h_i}$.
\end{proof}

\section{de Rham lifts by global methods}\label{sec:global
  lifts}\subsection{Potential automorphy and globalisation}In
this section, we make use of (global) potential automorphy techniques to produce
potentially crystalline lifts. Ultimately, these results rely on those
of~\cite{BLGGT}, but the actual global results we need are those of~\cite[App.\
A]{emertongeerefinedBM}.

The key idea is as follows: by the methods of~\cite{BLGGT} and~\cite{frankII},
we can often realise $\rbar:G_K\to\GL_n(\Fpbar)$ as the restriction to a decomposition
group of~$\rhobar$,
the reduction mod~$p$ of the $p$-adic Galois representation  associated to an
automorphic representation on some unitary group. Then the existence of congruences between
automorphic representations of different weights and types produces lifts
of~$\rbar$ of the corresponding Hodge and inertial types.

To keep this paper from becoming longer than necessary, and to avoid obscuring the relatively simple arguments
that we need to make, we will not recall the precise definitions of the spaces
of automorphic forms that we work with; the details may be found
in~\cite{emertongeerefinedBM} (and the papers referenced
therein). Suppose from now until the end of Lemma~\ref{lem:char 0
  versus char p} that:
\begin{itemize}
\item $p\nmid 2n$, and
\item $\rbar$ has a potentially diagonalisable lift of some type $(\lambda_{\rbar},\tau_{\rbar})$.
\end{itemize}
Then in particular Conjecture~A.3 of~\cite{emertongeerefinedBM} holds
for~$\rbar$, so that by~\cite[Cor.\ A.7]{emertongeerefinedBM}, there is a
CM field $F$ with maximal totally real field $F^+$, and an irreducible representation
$\rhobar:G_{F^+}\to\cG_n(\Fpbar)$ (where~$\cG_n$ is the algebraic group defined in~\cite[\S2.1]{cht}) which is automorphic in the sense of~\cite[Def.\
5.3.1]{emertongeerefinedBM}, and which globalises~$\rbar$ in the sense that
for each place $v\mid p$ of~$F^+$ we have that $v$ splits in $F$ and that
there is a place $\tv$ of $F$ lying over~$v$ such that $F_{\tv}\cong K$ and
$\rhobar|_{G_{F_{\tv}}}\cong\rbar$. The above data will remain fixed
throughout this section.

 Suppose that for each place $v\mid p$ of $F^+$ we fix a representation of
 $\GL_n(\cO_K)$ on a finite $\Zpbar$-module~$W_v$. Via the isomorphisms
 $\iota_{\tv}$ of~\cite[\S5.2]{emertongeerefinedBM}, we can regard
 $W:=\otimes_{\Zpbar,v\mid p}W_v$ as a representation of $G(\cO_{F^+,p})$, where~$G$ is
 a certain unitary group. Then there is a space of algebraic modular forms
 $S(U,W)$, as in~\cite[\S5.2]{emertongeerefinedBM}. (In fact, \cite{emertongeerefinedBM} works with coefficients in the ring of
 integers of some finite extension of~$\Qp$, rather than with
 $\Zpbar$-coefficients, but this makes no difference for the arguments we are
 making here.)

In particular, for any $(\lambda_v,\tau_v)_{v\mid p}$ a space of
automorphic forms $S_{\lambda,\tau}(U,\Zpbar)$ is defined
in~\cite[\S5.2]{emertongeerefinedBM} for certain sufficiently small
compact open subgroups
$U\subset G(\A_{F^+}^\infty)$ which are hyperspecial at $p$, 
corresponding to taking
each~$W_v$ to be~$\sigma(\lambda_{v},\tau_{v})$. Examining the proof of~\cite[Cor.\ A.7]{emertongeerefinedBM}, we see
that in fact we have $S_{\lambda_{\rbar},\tau_{\rbar}}(U,\Zpbar)_{\m}\ne 0$, where $\m$ is as
in~\cite[Def. 5.3.1]{emertongeerefinedBM}, and in a mild abuse of
notation we write
$(\lambda_{\rbar,v},\tau_{\rbar,v})=(\lambda_{\rbar},\tau_{\rbar})$ for all $v\mid p$. (This says that there is an automorphic
representation $\pi$ of weight~$\lambda_{\rbar}$ and type~$\tau_{\rbar}$, whose associated $p$-adic
Galois representation $\rho_\pi$ lifts~$\rhobar$; this representation
$\rho_\pi$ is the representation $\rho$ constructed in~\cite[Lem.\
A.5]{emertongeerefinedBM}.)

\begin{lem}
  \label{lem:char 0 versus char p}
Keep the notation and assumptions of the preceding discussion.

{\upshape(1)} If for some choice of
  $(\lambda_v,\tau_v)_{v\mid p}$ we have $S_{\lambda,\tau}(U,\Zpbar)_{\m}\ne 0$,
  then for each $v\mid p$, $\rbar$ has a potentially crystalline lift of
  type~$(\lambda_v,\tau_v)$.

{\upshape(2)} $S_{\lambda,\tau}(U,\Zpbar)_{\m}\ne 0$ if and only if there are Serre weights
$F_v$ of $\GL_n(k)$ such that
\begin{itemize}
\item $S(U,\otimes_{\Fpbar,v\mid p}F_{v})_{\m}\ne 0$, and
\item for all $v\mid p$, $F_{v}$
  is a Jordan--H\"older factor of $\sigmabar(\lambda_v,\tau_v)$.
\end{itemize}

\end{lem}
\begin{proof}
  (1) is immediate from~\cite[Prop.\ 5.3.2]{emertongeerefinedBM}. (We remind the reader
  that $S_{\lambda,\tau}(U,\Zpbar)$ is torsion-free.) In the case
  that $\tau$ is trivial, (2) is~\cite[Lem.\ 2.1.6]{blggun}, and the proof goes
  over unchanged to the general case.
\end{proof}

\begin{thm}\label{thm:pot diag implies serre weight}
  Suppose that $p\nmid 2n$, and that $\rbar$ has a potentially diagonalisable
  lift of some regular weight. Then the following hold.
  \begin{enumerate}
  \item There exists a finite extension $K'/K$ {\upshape(}depending only on $n$
    and $K$, and not on $\rbar${\upshape)} such that $\rbar$ has a lift $r : G_K
    \to \GL_n(\Zpbar)$ of Hodge type $0$ that becomes crystalline over $K'$.

  \item The representation $\rbar$ has a crystalline lift of some
  Serre weight.
  \end{enumerate}
\end{thm}
\begin{proof}
We begin with the proof of (2), since the argument is much shorter.
Let $r$ be the given potentially diagonalisable lift, and as above, write
  $(\lambda_{\rbar},\tau_{\rbar})$ for the type of~$r$. By Lemma~\ref{lem:char 0
    versus char p}(2), there are Jordan--H\"older factors $F_{a_v}$ of
  $\sigmabar(\lambda_{\rbar},\tau_{\rbar})$ (possibly varying with~$v$) such that
  $S(U,\otimes_{\Fpbar,v\mid p}F_{a_v})_{\m}\ne 0$. Let $\lambda_v$ be a lift of
  $a_v$ for each $v$, and let $\tau_v$ be trivial for each~$v$. Applying Lemma~\ref{lem:char 0
    versus char p}(2) again, we see that $S_{\lambda,1}(U,\Zpbar)_{\m}\ne 0$. By  Lemma~\ref{lem:char 0
    versus char p}(1), $\rbar$ has a crystalline lift of Hodge type~$\lambda_v$ for
  each $v\mid p$; any such lift will do.

Turn now to (1). 
As in the previous part we get
$S(U,V)_{\m} \neq 0$ for some irreducible representation $V =
\otimes_{v \mid p} V_v$ of $G(\cO_{F^+} \otimes \Zp)$ over $\Fpbar$.
Let $T \subset B \subset \GL_n$ denote the subgroups of diagonal and upper-triangular matrices,
as algebraic groups over $\Z$.
Consider $V_v$ as a representation of $\GL_n(\cO_{F_{\tilde v}})=:
K_v$ via $\iota_{\tilde v}$. Let $I_v \subset K_v$ denote the preimage
of $B(k_{\tilde v}) \subset \GL_n(k_{\tilde v})$. Then we can choose a
character $\chibar_v : I_v \to \Fpbar^{\times}$ such that $V_v |_{I_v}
\onto \chibar_v$. 

Let $q := \# k$.
We claim that for any $s\ge 1$ such that $q^{s-1} \ge n$, we can find
a (smooth) lift $\chi_v = \chi_{1,v} \otimes \cdots \otimes
\chi_{n,v} : T(\cO_{F_{\tilde v}}) \to \Zpbar^{\times}$ of $\chibar_v|_{T(\cO_{F_{\tilde v}})} = \chibar_{1,v}
\otimes \cdots \otimes \chibar_{n,v}$ such that the
$\{\chi_{i,v}\}_{i=1}^n$ are pairwise distinct and
$\chi_{i,v}|_{1+\varpi_{\tilde v}^s \cO_{F_{\tilde v}}} = 1$ for all $i$. Indeed,
recalling that $F_{\tilde v} \cong K$, write
$\cO_{F_{\tilde v}}^{\times}/(1 + \varpi_{\tilde v}^s \cO_{F_{\tilde v}}) \cong k^{\times} \times H$ (via the Teichm\"uller splitting), where
$H$ is abelian of order $q^{s-1}$. Then each $\chibar_{i,v}|_{k^\times}$ lifts uniquely to $\Zpbar^{\times}$, whereas $\chibar_{i,v}|_H = 1$
and can be lifted arbitrarily to $\Zpbar^{\times}$. Hence it is enough to note that $\#\Hom(H,\Zpbar^{\times})
= \#H = q^{s-1} \ge n$, and this proves the claim. For the rest
of the proof, we fix such a choice of $s$ and $\chi_v$.

Now, \cite[\S 3]{MR1621409} (applied with a standard Chevalley basis such that
$U_{\alpha,0} = U_{\alpha} \cap \GL_n(\cO_{F_{\tilde v}})$ for all
roots $\alpha$) provides a pair $(J_{\chi_v},\rho_{\chi_v})$
consisting of a compact open subgroup $J_{\chi_v} \subset I_v$ that
contains $T(\cO_{F_{\tilde v}})$ and a smooth character $\rho_{\chi_v}
: J_{\chi_v} \to \Zpbar^{\times}$ such that $\rho_{\chi_v}
|_{T(\cO_{F_{\tilde v}})} = \chi_v$. By construction, $\overline\rho_{\chi_v}$ is the restriction
of $\chibar_v$ to $J_{\chi_v}$, so by Frobenius reciprocity we get
a $K_v$-equivariant map $V_v \into \Ind_{J_{\chi_v}}^{K_v}(\rhobar_{\chi_{v}})$.

In particular, 
$S\left(U,\otimes_{v\mid p} \left(\Ind_{J_{\chi_v}}^{K_v} \rho_{\chi_v}\right)\right)_{\m}\neq 0$.
Using Deligne--Serre lifting we
get an automorphic representation $\pi$ of $G(\A_{F^+})$ with associated Galois
representation $\rho_{\pi}$ lifting $\rhobar|_{G_F}: G_F \to
\GL_n(\Fpbar)$ 
such that (i)
$\pi_\infty \cong \mathbf{1}$ and (ii)
$\Hom_{K_v}(\Ind_{J_{\chi_v}}^{K_v} \rho_{\chi_v}^{-1},\pi_v ) \neq 0$ (again via $\iota_{\tilde v}$) for any
$v \mid p$. 
Applying \cite[Thm~7.7]{MR1621409} (noting there are no
restrictions on $p$, cf.\ \cite[Lem.~3.1.6]{cht}), we deduce that
$\pi_v$ is a subquotient of $\Ind_{B(F_{\tilde v})}^{G(F_{\tilde v})}(\tilde \chi_v^{-1})$ for
some $\tilde\chi_v:T(F_{\tilde v}) \to \Qpbar^{\times}$ extending $\chi_v$. (Note
that $J_{\chi_v} = J_{\chi_v^{-1}}$.) Since the characters $\{\chi_{i,v}\}_{i=1}^n$ are pairwise distinct, the
Bernstein--Zelevinsky irreducibility criterion implies that $\pi_v \cong \Ind_{B(F_{\tilde v})}^{G(F_{\tilde v})}(\tilde
\chi_v^{-1})$.

It follows that $\mathrm{rec}(\pi_v)$ has $N=0$ and on inertia is of the form
$\chi^{-1}_{1,v} \oplus \cdots \oplus \chi^{-1}_{n,v}$ via the local Artin map, where 
$\mathrm{rec}$ denotes the local Langlands correspondence, normalised
as in \cite{emertongeerefinedBM} (i.e., as in \cite{ht}). 
Using Lemma~\ref{lem:localfields} below there exists a finite extension $K'/K$
depending only on $n$ and $K$ such that $\mathrm{rec}(\pi_v)|_{I_{K'}}$ is
trivial. Applying local-global compatibility at $p$ to $\rho_{\pi}$,
we deduce that
for any $v \mid p$, the representation $\rho_{\pi}|_{G_{F_{\tv}}}$ provides a desired lift of $\rhobar|_{G_{F_{\tv}}}\cong\rbar$.
\end{proof}

\begin{remark}
  The argument in the proof above shows that if $\chi_{i,v} : \cO_{F_v}^\times \to \Qpbar^\times$ are pairwise distinct smooth characters
of $\cO_{F_v}^\times$ (or equivalently of $I_{F_v}$), then $\Ind_{J_{\chi_v}}^{K_v} \rho_{\chi_v}$ is a $K_v$-type corresponding to
$\oplus_{i=1}^n \chi_{i,v}$ under the inertial Langlands correspondence, i.e.\
\cite[Conj.~4.1.3]{emertongeerefinedBM} holds with $\sigma(\oplus_{i=1}^n \chi_{i,v}) \cong \Ind_{J_{\chi_v}}^{K_v} \rho_{\chi_v}^{-1}$.
\end{remark}

\begin{lem}\label{lem:localfields}
Suppose $K/\Qp$ is a finite extension and $s \ge 1$. Then there
exists a finite extension $L/K$ such that any smooth character $\chi :
W_K \to \mathbb{C}^\times$ that is trivial on the ramification subgroup
$G_K^s$ satisfies $\chi|_{I_L} = 1$.
\end{lem}

\begin{proof}
By local class field theory there exists a finite extension
$M_s/K^{\nr}$ that is independent of $\chi$ such that $\chi|_{G_{M_s}} = 1$. (We
can take $M_s/K$ abelian such that $K^{\mathrm{ab}}/M$ has Galois group
$1+\varpi_K^s \cO_K$, with $\varpi_K$ a uniformiser of $K$.) Then we choose $L/K$ finite such that $M_s$ is contained in
$ L \cdot K^{\nr} = L^{\nr}$. This implies $\chi|_{I_L} = 1$. In fact, this argument
shows that we can take $L/K$ of degree $q^{s-1}(q-1)$, where $q = \#k$.
\end{proof}

Our final result may be viewed as a ``weak Breuil--M\'ezard''-type statement.

\begin{thm}
  \label{thm:FL lift implies lots more lifts}Suppose that $p\ne n$,
 that $K/\Qp$ is
  unramified, and that $\rbar$ has a crystalline lift
  of  weight~$F$ for some extended FL weight~$F$. 
  If $F$ is a
  Jordan--H\"older factor of~$\sigmabar(\lambda,\tau)$ for some $\lambda$,
  $\tau$, then $\rbar$ has a potentially crystalline lift of type~$(\lambda,\tau)$.
\end{thm}
\begin{proof} Choose $a \in(X_1^{(n)})^{\Hom(k,\Fpbar)}$ such that
  $F \cong F_a$.  The conditions that $p\ne n$ and $\rbar$ has a crystalline lift
  of  weight~$F_a$ with~$a$ an extended FL weight imply that $p >
  n$; so either $p\nmid 2n$, as we have assumed throughout this
  section, or else $p=2$ and $n=1$.

  First suppose that $p \nmid 2n$. By the main result of~\cite{GaoLiu12} any crystalline
  representation of extended
  FL weight is potentially diagonalisable. Let $\lambda'$ be the
  lift of $a$ (uniquely defined, as $K/\Qp$ is unramified). Since
  $a_{\kappabar,1}-a_{\kappabar,n}\le p-(n-1)$ 
  for each~$\kappabar$,
  $F_a=L_a\otimes_{\Zpbar}\Fpbar$ (see \S\ref{sec:repr-gl_n-serre}). By hypothesis, we can apply the 
  constructions from the paragraphs preceding Lemma~\ref{lem:char 0
    versus char p}  with $\lambda_{\rbar}=\lambda'$ and $\tau_{\rbar}=1$ to deduce that
  $S_{\lambda',1}(U,\Zpbar)_{\m}\ne 0$. By
  Lemma~\ref{lem:char 0 versus char p}(2),
  $S(U,\otimes_{\Fpbar,v|p}F_{a})_{\m}\ne 0$.   Applying Lemma~\ref{lem:char 0 versus char p}(2) with
  $(\lambda_v,\tau_v)=(\lambda,\tau)$ for each~$v$, we see that
  $S_{\lambda,\tau}(U,\Zpbar)_{\m}\ne 0$, and the result follows from
  Lemma~\ref{lem:char 0 versus char p}(1).

On the other hand, the case $n=1$ is an easy consequence of local class field
theory:\ $\sigma(\tau)^\vee$ is obtained from $\tau$ by local class field
theory, so that the locally algebraic
characters of $K^{\times}$ extending $\sigma(\lambda,\tau)$ correspond to de Rham
characters of type $(\lambda,\tau)$, while $\rbar|_{I_K}$ corresponds to $F_a$.
\end{proof}

\bibliographystyle{amsalpha} 
\bibliography{FLlifts} 

\end{document}